\documentclass[12pt]{amsart}

 \usepackage{amsfonts,graphics,amsmath,amsthm,amsfonts,amscd,amssymb,amsmath,latexsym,multicol,
 mathrsfs}
\usepackage{epsfig,url}
\usepackage{flafter}
\usepackage{fancyhdr}
\usepackage{hyperref}
\hypersetup{colorlinks=true, linkcolor=black}
\makeatletter

\def\jobis#1{FF\fi
  \def\predicate{#1}%
  \edef\predicate{\expandafter\strip@prefix\meaning\predicate}%
  \edef\job{\jobname}%
  \ifx\job\predicate
}

\makeatother

\if\jobis{proposal}%
\else
\fi

 \usepackage[matrix, arrow]{xy}

\newtheorem{lemma}{Lemma}[section]
\newtheorem{thm}[lemma]{Theorem}
\newtheorem{lem}[lemma]{Lemma}
\newtheorem{cor}[lemma]{Corollary}
\newtheorem{prop}[lemma]{Proposition}

\newtheorem{conj}[lemma]{Conjecture}

\theoremstyle{definition}
\newtheorem{defn}[lemma]{Definition}
\newtheorem{rem}[lemma]{Remark}

\newtheorem{exa}[lemma]{Example}
\newtheorem{quest}[lemma]{Question}

\theoremstyle{remark}
\newtheorem*{remark*}{Remark}
\newtheorem*{note*}{Note}

\title{Valuative multiplier ideals}
\author{Zhengyu Hu}
\date{\today}
\begin{document}
\maketitle

\begin{abstract}
The main goal of this paper is to construct an algebraic analogue of quasi-plurisubharmonic function (qpsh for short) from complex analysis and geometry. We define a notion of qpsh function on a valuation space associated to a quite general scheme. We then define the multiplier ideals of these functions and prove some basic results about them, such as subadditivity property, the approximation theorem. We also treat some applications in complex algebraic geometry.
\end{abstract}

%\tableofcontents

%%%%%%%%%%%%%%%%%%%%%%%%
%%%%%%%%%%%%%%%%%%%%%%%%%%

\section{Introduction}

Given a line bundle $L$ on a smooth projective complex variety, a classical theorem of Kodaira asserts that if $L$ carries on a smooth metric with positive curvature, then $L$ is ample, or equivalently the global sections of a multiple of $L$ give an embedding to a projective space and hence induce such a metric on $L$. More generally, global sections of a multiple of $L$ induce a semi-positive singular metric. Conversely, given a semi-positive singular metric $h$, the local weight function $\varphi$, which is plurisubharmonic (psh for short), should be related to sections of multiples of $L$, or perhaps of a small perturbation of $L$. See [\ref{Lehmann}] for more details.

On the other hand, if we work locally near the origin of $\mathbb{C}^n$, then [\ref{BFJ1}, Section 5] shows that we can transform a psh germ $\varphi$ to a formal psh function $\widehat{\varphi}$ on quasi-monomial valuations centered at the origin. This valuative transform usually loses much information on the original psh function, however, it preserves the information on the singularity of $\varphi$. In particular, they give the same multiplier ideals which essentially means that they characterize the same singularity because of the Demailly's approximation. The idea of studying psh functions using valuations was systematically developed in [\ref{BFJ1}] and its predecessors [\ref{FJ1}], [\ref{FJ2}], [\ref{FJ3}]. The main purpose of this paper is to define a similar notion of qpsh functions on a separated, regular, connected and excellent schemes over $\mathbb{Q}$, and we then study these functions.

Although we don't discuss Berkovich spaces in this paper, our work should be related to the qpsh functions (or metrics on line bundles) on the Berkovich space associated to a smooth projective variety over a trivially valued field. See [\ref{BFJ2}] and [\ref{BFJ3}].

Let us briefly introduce some terminologies. Roughly speaking, we consider a function $\varphi$ on divisorial valuations on a scheme $X$ such that $\varphi(t \mathrm{ord}_E)=t \varphi(\mathrm{ord}_E)$ and $\sup_E \frac{|\varphi(\mathrm{ord}_E)|}{A(\mathrm{ord}_E)}<+\infty$ where $E$ runs over all prime divisors over $X$. We prove that such functions form a Banach space $\mathrm{BH}(X)$ if we equip it with the norm $\|\varphi\|=\sup_E \frac{|\varphi(\mathrm{ord}_E)|}{A(\mathrm{ord}_E)}$ (see Proposition \ref{norm_reg}). By convention we set $\log|\mathfrak{a}|(\mathrm{ord}_E)=-\mathrm{ord}_E(\mathfrak{a})$ for a nonzero coherent ideal sheaf $\mathfrak{a}$, and one can easily check that $\log|\mathfrak{a}|$ is a valuative function in $\mathrm{BH}(X)$. We define the set of qpsh function $\mathrm{QPSH}(X)$ to be the closed convex cone generated by functions of the form $\log|\mathfrak{a}|$. We then define the multiplier ideal $\mathcal{J}(\varphi)$ of a qpsh function $\varphi$ to be the largest ideal $\mathfrak{a}$ such that $\sup_E \frac{-\mathrm{ord}_E(\mathfrak{a})-\varphi(\mathrm{ord}_E)}{A(\mathrm{ord}_E)} <1$. This definition is reasonable because of Proposition \ref{prop:ideal_mult} and Corollary \ref{cor:alg-mult-ideal}.

Our first main result is that a qpsh function is a decreasing limit of a sequence of qpsh functions of the form $c_k\log|\mathfrak{b}_k|$. In complex analysis and geometry, such a regularization is crucial. See [\ref{Dem1}], [\ref{Dem2}]. Moreover, we prove that we can actually choose $\mathfrak{b}_k=\mathcal{J}(k\varphi)$ which satisfies the subadditivity property. See Proposition \ref{subadditivity}(1). Readers can compare this result with [\ref{DEL}].

\begin{thm}[cf. Theorem \ref{ch_qpsh}]
Let $\varphi$ be a bounded homogeneous function. Then $\varphi$ is qpsh if and only if $\varphi$ is the limit function, in the norm, of a decreasing sequence of qpsh functions of the form $c_k\log|\mathfrak{b}_k|$. Furthermore, we can choose $c_k=\frac{1}{k}$ and $\mathfrak{b}_k=\mathcal{J}(k\varphi)$ which form a subadditive sequence of ideals.
\end{thm}

Given an ideal $\mathfrak{a}$ on a scheme $X$, the log canonical threshold $\mathrm{lct}(\mathfrak{a})$ is a fundamental invariant both in singularity theory and birational geometry (see [\ref{Laz}], [\ref{KM}], etc.). The log canonical threshold admits the following description in terms of valuations:
$$
\mathrm{lct}(\mathfrak{a})=\inf_E \frac{A(\mathrm{ord}_E)}{\mathrm{ord}_E(\mathfrak{a})}
$$
where $E$ runs over all prime divisors over $X$ and $A(\mathrm{ord}_E)=\mathrm{ord}_E (K_{Y/X})+1$. In fact in the above formulae one can take the infimum over all real valuations centered on $X$. It is well-known that if $Y$ is a log resolution of $\mathfrak{a}$, then there exists some prime divisor $E$ on $Y$ such that $\mathrm{ord}_E$ computes the log canonical threshold, that is, $\mathrm{lct}(\mathfrak{a})=\frac{A(\mathrm{ord}_E)}{\mathrm{ord}_E(\mathfrak{a})}$. Given a qpsh function $\varphi$, we can define the log canonical threshold $\mathrm{lct}(\varphi)$ as the limit of $\frac{1}{c_k}\mathrm{lct}(\mathfrak{a}_k)$ where $c_k\log|\mathfrak{a}_k|$ converges to $\varphi$ strongly in the norm. We show that
$$
\mathrm{lct}(\varphi)=\inf_E \frac{A(\mathrm{ord}_E)}{-\varphi(\mathrm{ord}_E)}.
$$
Unfortunately, there might be no divisorial valuation which computes the log canonical threshold in general. However, we can prove that there exists a real valuation which computes the log canonical threshold. This has been heavily studied in [\ref{JM}], [\ref{JM2}] and other references. It is conjectured by [\ref{JM}, Conjecture B] that a valuation which computes the norm is quasi-monomial (see Conjecture \ref{conj}). Equivalently we consider the reciprocal of the log canonical threshold which is exactly the norm of $\varphi$ by definition. More generally, for a nonzero ideal $\mathfrak{q}$ we consider $\|\varphi\|_{\mathfrak{q}}:=\sup_E \frac{-\varphi(\mathrm{ord}_E)}{A(\mathrm{ord}_E)+\mathrm{ord}_E(\mathfrak{q})}$, and we prove that there exists a real valuation which compute this norm. The proof in this paper mainly follows the strategy of [\ref{JM}]. A similar result appears in [\ref{JM2}].

\begin{thm}[=Theorem \ref{c_norms}]
Let $\varphi\in \mathrm{QPSH}(X)$ be a qpsh function and let $\mathfrak{q}$ be a nonzero ideal on $X$. Then there exists a nontrivial tempered valuation $v$ which computes $\|\varphi\|_{\mathfrak{q}}$.
\end{thm}

If $X$ is a complex projective variety, then we can provide $\mathrm{QPSH}(X)$ with more structures. Namely, given a $\mathbb{Q}$-line bundle $L$ on $X$, we say that the function $\lambda \log|\mathfrak{a}|$ is $L$-psh if $\lambda$ is a nonnegative rational number and $L \otimes \mathfrak{a}^\lambda$ is semi-ample. We can then define $\mathrm{PSH}(L) \subseteq \mathrm{QPSH}(X)$ as the closure of the set of such functions. We also define the set of pseudo $L$-psh functions as $\mathrm{PSH}_\sigma (L) := \bigcap_{\epsilon>0} \mathrm{PSH}(L +\epsilon A)$, where $A$ is an ample line bundle. See Section \ref{sec:D-psh} for more details.

Under the above setting, we show that there exists the maximal $L$-psh function $\varphi$ which can be written explicitly as $\varphi(v)=-v(\|L\|)$, and there exists the maximal pseudo $L$-psh function $\phi$ which can be written explicitly as $\phi(v)=-\sigma_v(\|L\|)$ (see Proposition \ref{prop:max} and \ref{prop:max'}). As an immediate corollary we generalize [\ref{Lehmann}, Theorem 6.14] as follows. See [\ref{Lehmann}] for the definition of the perturbed ideal and the diminished ideal.

\begin{thm}[=Theorem \ref{c_max_norm}]
Let $D$ be a pseudo-effective divisor. Assume that $\phi_{\max}$ is the maximal pseudo $D$-psh function. Then, the perturbed ideal $\mathcal{J}_{\sigma,-}(D)=\mathcal{J}_{-}(\phi_{\max})$, and the diminished ideal $\mathcal{J}_\sigma(D)=\mathcal{J}(\phi_{\max})$. In particular, we can write $\mathcal{J}_\sigma(D)$ explicitly as $\Gamma(U,\mathcal{J}_{\sigma}(L))=\{f\in \Gamma(U,\mathcal{O}_{X})|v(f)+A(v)-\sigma_{v}(\|L\|)>0$ for all $v\in \mathrm{V}_U^\ast\}$. Further, a nonzero ideal $\mathfrak{q} \subseteq \mathcal{J}_{\sigma}(\|L\|)$ if and only if $v(\mathfrak{q})+A(v)-\sigma_{v}(\|L\|)>0$ for all $v\in \mathrm{V}_{X}^{\ast}$.
\end{thm}

In the last subsection of this paper, we prove the finite generation of a divisorial module as another application. The proposition below can also be obtained using minimal model theory (see Remark \ref{rem:mmp}). Note that our proof here avoids using 'the length of extremal rays' (see [\ref{BH-I}]).

\begin{prop}[=Proposition \ref{prop:f_g}]
Let $(X,B)$ be a log canonical pair. Assume that $K_X+B$ is $\mathbb{Q}$-Cartier and abundant, and that $R(K_X+B)$ is finitely generated. Then, for any reflexive sheaf $\mathscr{F}$, $M_{\mathscr{F}}^p(K_X+B)$ is a finitely generated $R(K_X+B)$-module.
\end{prop}

From an easy observation, the above proposition can be slightly generalized (cf. Proposition \ref{prop:f_g_g}). \\

\textbf{Acknowledgements.} This paper is based on part of the author's PhD thesis. I would like to express my deep gratitude to my supervisor Professor Kefeng Liu for numerous conversations and encouragement from him. I am indebted to Professor Hongwei Xu and many other professors in Zhejiang Unversity because I have greatly benefited from the courses and discussions with them. I would also like to thank Professor S\'{e}bastien Boucksom, Mattias Jonsson and Mircea Musta\c{t}\u{a} for patiently answering my questions and providing many valuable comments. The author is supported by the EPSRC grant EP/I004130/1.

%%%%%%%%%%%%%%%%%%%%%%%%%%%%%%%%%%%%
\section{Valuation spaces}\label{section:val_space}

Throughout this paper, all schemes are assumed to be separated, regular, connected and excellent schemes over $\mathbb{Q}$. All rings are assumed to be integral, regular and excellent rings containing $\mathbb{Q}$. An ideal on a scheme means a coherent ideal sheaf on a scheme. A birational model of a scheme is a scheme birational to and proper over this scheme, and a divisor over a scheme is a divisor on a birational model of the scheme. For definitions and properties of valuations, multiplier ideals, singularities in birational geometry, etc., we refer to [\ref{Laz}],[\ref{JM}] and [\ref{KM}]. \\

\textbf{Real valuations.}
Let $X$ be a scheme, and let $K(X)$ be its function field. A \emph{real valuation} $v$ is a function
$v:K(X)^{\ast}\longrightarrow \mathbb{R}$ such that $v(fg)=v(f)+v(g)$ and $v(f+g)\geq \min\{v(f),v(g)\}$. By convention we set $v(0):=+\infty$. Let $\mathcal{O}_{v}:=\{f|v(f)\geq 0\}$ be its valuation ring. If there exists a point $\xi\in X$ such that the
 morphism $\mathcal{O}_{X,\xi}\hookrightarrow \mathcal{O}_{v}$ is a local homomorphism, then $\xi$ is called the \emph{centre}
 of $v$ on $X$ and denoted by $c_{X}(v)$. Note that $\xi$ is unique since $X$ is separated, and also note that the centre always exists provided that $X$ is complete. A real valuation with centered on $X$
  is called a real valuation on $X$ or simply a valuation on $X$, and we denote by $\mathrm{Val}_{X}$ the set of valuations on $X$. The set of valuations $\mathrm{Val}_{X}$ is independent of the choice of a birational model of $X$. More precisely, if $Y \rightarrow X$ is a proper birational morphism of schemes, then $\mathrm{Val}_X=\mathrm{Val}_Y$. A valuation $v$ on $X$ is said to be the \emph{trivial} valuation if its centre $c_{X}(v)$ is the generic point of $X$. We denote by $\mathrm{Val}_{X}^{\ast} \subseteq \mathrm{Val}_X$ the set of nontrivial valuations on $X$.

The set $\mathrm{Val}_X$ can be equipped with an induced topology defined by the maps $v\longrightarrow v(f)$ for all rational functions $f \in K(X)^\ast$. For every nonzero ideal $\mathfrak{a}$, we have that $v(\mathfrak{a})$ is well defined and $v(\mathfrak{a})=v(\overline{\mathfrak{a}})$ where $\overline{\mathfrak{a}}$ denotes the integral closure of $\mathfrak{a}$. Note that the topology on $\mathrm{Val}_X$ defined by pointwise convergence on ideals on $X$ is equivalent to that on functions in $K(X)$. Readers can consult [\ref{JM}, Section 1] for more details.

Under the above topology, the map $c_{X}: \mathrm{Val}_{X}\longrightarrow X$ is anti-continuous. That is, the inverse image of an open subset is closed.
More precisely, if $U\subseteq X$ is an open subset and $\mathfrak{m}$ is the defining ideal of $X\setminus U$, then $\mathrm{Val}_{U}=\{v\in\mathrm{Val}_{X}|v(\mathfrak{m})=0\}$ and $\mathrm{Val}_{U}$ is closed in $\mathrm{Val}_{X}$.

For two valuations $v$, $w$ on $X$, we say that $v\leq w$ if $v(\mathfrak{a})\leq w(\mathfrak{a})$ for every nonzero ideal $\mathfrak{a}$. This is equivalent to that the centre
 $\eta:= c_{X}(w)\in \overline{c_{X}(v)}$ and that $v(f)\leq w(f)$ for every nonzero local function $f\in \mathcal{O}_{X,\eta}$.\\

\textbf{Quasi-monomial valuations.}
Let $X$ be a scheme, let $\xi\in X$ be a point, and let $\underline{x}=(x_{1},\ldots,x_{r})$ be a regular system of parameters at $\xi$. If $f\in \mathcal{O}_{X,\xi}$ is a local regular function, then
$f$ can be expressed as $f=\sum_{\beta} c_{\beta}x^{\beta}$ in $\widehat{\mathcal{O}_{X,\xi}}$ with each coefficient $c_{\beta}$ either zero or a unit. For each $\alpha=(\alpha_{1},\ldots,\alpha_{r})\in \mathbb{R}_{\geq 0}^{r}$, we define a real valuation by $\mathrm{val}_{\xi,\alpha}(f)=\min \{<\alpha,\beta>|c_{\beta}\neq 0\}$ where $<\alpha,\beta>:=\sum_{i} \alpha_{i} \beta ^{i}$, which is called a \emph{monomial} valuation on $X$.

A pair $(Y,D)$ is called log smooth if $Y$ is a scheme and $D$ is a reduced divisor whose components are regular subschemes intersecting each other transversally. A pair $(Y,D)$ is called a log resolution of $X$ if there is a birational projective morphism $\pi:Y \rightarrow X$ and $(Y,D+K_{Y/X})$ is log smooth. Let $(Y',D')$ be another log resolution of $X$, we say $(Y',D')\succeq (Y,D)$ if $Y'$ is projective over $Y$ and the support of $D'$ contains the support of the pull-back of $D$. Note that log resolutions of $X$ form an inverse system.

Let $(Y,D)$ be a log resolution of $X$, and let $\eta$ be the generic point of an irreducible component of the intersection of some prime components of $D$. We denote by $\mathrm{QM}_{\eta}(Y,D)$ the set of real valuations which can be defined as a monomial valuation at $\eta$. Note that $\eta \in \overline{c_{X}(v)}$ and $\mathrm{QM}_{\eta}(Y,D)\cong \mathbb{R}_{\geq 0}^{r}$ as topological spaces. We also define $\mathrm{QM}(Y,D)=\bigcup _{\eta}\mathrm{QM}_{\eta}(Y,D)$ where $\eta$ runs over every generic point of some component of the intersection of some prime components of $D$. A real valuation $v$ is said to be \emph{quasi-monomial} if there exists a log resolution $(Y,D)$ such that $v \in \mathrm{QM}(Y,D)$.

\begin{rem}
Let $\Gamma_{v}=v(K(X)^{\ast})\subseteq \mathbb{R}$ be the value group of $v$. Let $\mathrm{ratrk}(v)=\dim_{\mathbb{Q}}(\Gamma_{v}\otimes _{\mathbb{Z}}\mathbb{Q})$ be the rational rank of $v$, and let $k_{v}$, $k(\xi)$ be the residue fields $\mathcal{O}_{v}$, $\mathcal{O}_{X,\xi}$ respectively, where $\xi=c_{X}(v)$. If we denote by $\mathrm{trdeg}_{X}v=\mathrm{trdeg}(k_{v}/k(\xi))$ the transcendental degree of $v$ over $X$, we have the Abyankar inequality $\mathrm{ratrk}(v)+\mathrm{trdeg}_{X}v\leq \dim(\mathcal{O}_{X,\xi})$. A result asserts that the quasi-monomial valuations are exactly the ones that give equality in the Abhyankar inequality (cf. [\ref{JM}, Proposition 3.7]).
\end{rem}

Let $v\in \mathrm{Val}_{X}$ be a quasi-monomial valuation. A log smooth pair $(Y,D)$ is said to be adapted to $v$ if $v\in \mathrm{QM}(Y,D)$. We say $(Y,D)$ is a good pair adapted to $v$ if $\{ v(D_{i}) | v(D_{i}) >0 \}$ are rationally independent. The following useful lemma is established as [\ref{JM}, Lemma 3.6].

\begin{lem}
Let $v\in \mathrm{Val}_{X}$ be a quasi-monomial valuation. There exists a good pair $(Y,D)$ adapted to $v$. If $(Y',D')\succeq (Y,D)$ and $(Y,D)$ is a good pair adapted to $v$, then $(Y',D')$ is also a good pair adapted to $v$.
\end{lem}

An important class of valuations are divisorial valuations. A valuation is called \emph{divisorial} if it is positively proportional to $\mathrm{ord}_{E}$ for some prime divisor $E$ over $X$, where $\mathrm{ord}_{E}$ is the vanishing order along $E$. One easily verifies that the trivial valuation is quasi-monomial of rational rank zero, and a divisorial valuation is quasi-monomial of rational rank one. Let $(Y,D)$ be a log smooth pair adapted to $v$. It can be verified that $v$ is divisorial if and only if $\mathbb{R}_{\geq 0}[v]\subseteq \mathrm{QM}_{\eta}(Y,D)\cong \mathbb{R}_{\geq 0}^{r}$ is a rational ray, that is, $\mathbb{R}_{\geq 0}[v]$ contains some rational point in $\mathbb{R}_{\geq 0}^{r}$.

For every log resolution $(Y,D)$ we can define the retraction map
$$
r_{Y,D}:\mathrm{Val}_{X} \longrightarrow \mathrm{QM}(Y,D)
$$
by taking $v$ to a quasi-monomial valuation in $\mathrm{QM}(Y,D)$ with $r_{Y,D}(v)(D_{i})=v(D_{i})$. Note that $r_{Y,D}$ is continuous and $v\geq r_{Y,D}(v)$ with equality if and only if $v\in \mathrm{QM}(Y,D)$. Furthermore, if $(Y',D')\succeq (Y,D)$ is another resolution, then the retraction map $r_{Y,D}:\mathrm{QM}(Y',D')\longrightarrow \mathrm{QM}(Y,D)$ (by abuse of notaion if without confusion) is a surjective mapping which is integral linear on every $\mathrm{QM}_{\eta'}(Y',D')$ and we have that $r_{Y,D}\circ r_{Y',D'}=r_{Y,D}$. Therefore we can naturally regard $\mathrm{QM}(Y,D)$ as a subset of $\mathrm{QM}(Y',D')$, and hence of the set of quasi-monomial valuations on $X$. Also note that $v(\mathfrak{a}) \geq r_{Y,D}(v)(\mathfrak{a})$ for an ideal $\mathfrak{a}$ on $X$, with equality if $(Y,D)$ is a log resolution of $\mathfrak{a}$. (cf. [\ref{Laz}], [\ref{JM}, Corollary 4.8])

\textbf{Tempered valuations.}
We first introduce the log discrepancy on an arbitrary scheme. Let $\pi: Y \longrightarrow X$ be a birational proper morphism. The $0^{\mathrm{th}}$ fitting ideal $\mathrm{Fitt}_{0}(\Omega_{Y/X})$ is a locally principle ideal with its corresponding effective divisor denoted by $K_{Y/X}$ (cf. [\ref{JM}, Section 1.3]). For a quasi-monomial valuation $v\in \mathrm{QM}(Y,D)$, we define the log discrepancy
$$
A_{X}(v)=\sum v(D_{i})\cdot A_{X}(\mathrm{ord}_{D_{i}})=\sum v(D_{i})\cdot (1+ \mathrm{ord}_{D_{i}}(K_{Y/X})).$$
We simply denote this by $A$ when the scheme $X$ is obvious. Note that $A$ is strictly positive linear on every $\mathrm{QM}_\eta(Y,D)$, and in particular continuous on every $\mathrm{QM}_\eta(Y,D)$ (or is \emph{weakly continuous} according to Definition \ref{defn:val_fun}). Also note that if $(Y',D')\succeq (Y,D)$ and $v\in \mathrm{QM}(Y',D')$, then $A(v)\geq A(r_{Y,D}(v))$ and equality holds only when $v\in \mathrm{QM}(Y,D)$. For an arbitrary valuation $v\in \mathrm{Val}_{X}$, we define $$A(v)=\sup_{(Y,D)}A(r_{Y,D}(v)) \in [0,+\infty].$$
Note that $A$ is lower-semicontinuous (lsc) as a valuative function.

\begin{defn}
 A valuation $v$ is said to be \emph{tempered} if $A(v)<\infty$. The valuation space $\mathrm{V}_X$ of $X$ is defined to be the space of tempered valuations as a subspace of $\mathrm{Val}_{X}$.
\end{defn}

We similarly denote by $\mathrm{V}_{X}^{\ast}$ the subset of nontrivial tempered valuations. If $f:X' \rightarrow X$ is a proper birational morphism, then $A_X(v)=A_{X'}(v)+v(K_{X'/X})$ (cf. [\ref{JM}, Proposition 5.1(3)]) and hence $\mathrm{V}_{X'}=\mathrm{V}_X$. Since $\mathrm{V}_{X}$ is a topological subspace of $\mathrm{Val}_{X}$, it is naturally a subspace of of the Berkovich space $X^{an}$. See [\ref{JM}, Section 6.3] for a comparison.

With the aid of the log discrepancy, we can normalize $\mathrm{V}_X^\ast$ by letting $A(v)=1$, that is, we define $\Lambda_X:=\{v\in\mathrm{V}_X^\ast|A(v)=1\}$. In particular, we normalize every cone complex $\mathrm{QM}(Y,D)$ by setting $\Delta(Y,D):=\{v\in\mathrm{QM}(Y,D)|A(v)=1\}$. It is clear that $\Delta(Y,D)$ naturally possess the structure of a simplicial complex, and by convention we say that $\Delta(Y,D)$ is a \emph{dual complex}. Readers can compare the constructions here with [\ref{BFJ1}], [\ref{BFJ2}] and [\ref{BFJ3}].

The following lemma allows us to compare $v$ and $\mathrm{ord}_{\xi}$ where $\xi=c_{X}(v)$ which is quite useful (see [\ref{Laz}], [\ref{JM}, Section 5.3] for the definition of $\mathrm{ord}_{\xi}$). See [\ref{JM}, Proposition 5.10] for a proof. Recently S. Boucksom, C. Favre and M. Jonsson gave a refinement of the following lemma in [\ref{BFJ4}].

\begin{lem}[Izumi's type inequality]\label{Izumi}
Let $\xi=c_{X}(v)$ and $\mathfrak{m}_{\xi}$ be the defining ideal of $\overline{\{\xi\}}$. Then we have $v(\mathfrak{m}_{\xi})\mathrm{ord}_{\xi} \leq v\leq A(v)\mathrm{ord}_{\xi}$.
\end{lem}

\textbf{Passing to the completion.}
A morphism $f:X' \rightarrow X$ is \emph{regular} if it is flat and its fibres are geometrically regular (cf. [\ref{JM},Section 1.1]). The following lemma on log discrepancy is essential for finding a valuation which computes the log canonical threshold or norms in Section \ref{c-norms}.

\begin{lem}[\ref{JM}, Proposition 5.13]\label{disc_reg}
Let $f:X'\longrightarrow X$ be a regular morphism, and let $f_{\ast}: \mathrm{Val}_{X'}\longrightarrow \mathrm{Val}_{X}$ be the induced map. If $v'\in \mathrm{Val}_{X'}$ is a valuation on $X'$, then $A(v')\geq A(f_{\ast}(v'))$. If we assume further that $X'=\mathrm{Spec}\widehat{ \mathcal{O}_{X,\xi}}$ and $v'$ is centered at the closed point of $X'$, then $A(v')=A(f_{\ast}(v'))$.
\end{lem}

\begin{defn}\label{defn:cpt-valsp}
If $\xi \in X$ is a point, then we define $\mathrm{V}_{X,\xi}:=c_{X}^{-1}(\xi)$ as a subspace of $\mathrm{V}_{X}$. We can normalize $\mathrm{V}_{X,\xi}$ by letting $v(\mathfrak{m})=1$, where $\mathfrak{m}$ is the defining ideal of $\overline{\{\xi\}}$. More precisely we define $\mathbb{V}_{X,\xi}:=\{v\in \mathrm{V}_{X,\xi}|v(\mathfrak{m})=1 \}$. Let $M>0$ be a positive real number, we also define $\mathbb{V}_{X,\xi,M}:=\{ v\in \mathbb{V}_{X,\xi}| A(v) \leq M \}$. According to [\ref{JM}, Proposition 5.9] the space $\mathbb{V}_{X,\xi,M}$ is compact. If $X=\mathrm{Spec} A$ and $\mathfrak{m}$ is the defining ideal of $\overline{\{\xi\}}$, we often use the notation $\mathbb{V}_{A,\mathfrak{m}}$ instead of $\mathbb{V}_{X,\xi}$.
\end{defn}

Let $(R,\mathfrak{m})$ be a local ring. Given a tempered valuation $v \in \mathrm{V}_{R,\mathfrak{m}}$, we define $v'(f)=\lim\limits_{k \rightarrow \infty}v(\mathfrak{a}_{k})$ for every $f\in \widehat{R}$ where $\mathfrak{a}_{k}\cdot \widehat{R}=(f)+\widehat{\mathfrak{m}}^{k}$. This is well-defined since $v(\mathfrak{a}_{k})\leq A(v)\mathrm{ord}_{\xi}(\mathfrak{a}_{k})\leq A(v)\mathrm{ord}_{\xi'}(f)<\infty$ by Lemma \ref{Izumi}. The above definition leads to a correspondence between the valuation spaces of $\mathrm{Spec}R$ and $\mathrm{Spec}\widehat{R}$ as follows. Throughout this paper we will use the notations $v$ and $v'$ to indicate that $v=f_\ast v'$ for simplicity if without confusion.

\begin{prop}
Let $(R,\mathfrak{m})$ be a local ring, and let $(\widehat{R},\widehat{\mathfrak{m}})$ be its $\mathfrak{m}$-adic completion. If we denote by $f: \mathrm{Spec}\widehat{R}\longrightarrow \mathrm{Spec}R $ the canonical morphism, then the induced map $f_{\ast}: \mathbb{V}_{\widehat{R},\widehat{\mathfrak{m}}}\longrightarrow \mathbb{V}_{R,\mathfrak{m}}$ is bijective. If $K'$ is a compact subspace of $\mathbb{V}_{\widehat{R},\widehat{\mathfrak{m}}}$, then $f_{\ast}$ is a homeomorphism from $K'$ to its image. In particular, $ \mathbb{V}_{\widehat{R},\widehat{\mathfrak{m}},M} \cong \mathbb{V}_{R,\mathfrak{m},M}$ for any positive number $M>0$.
\end{prop}

\begin{proof}
The bijectivity of $f_\ast$ follows from [\ref{JM}, Corollary 5.11], and we will prove the latter statement. Let $K=f_{\ast}(K')$. It suffices to show that $K'$ is homeomorphic to $K$. Let $h \in \widehat{R}$ be a nonzero function. We have that $\max_{v'\in K'} v'(h)= \alpha < \infty$ since $K'$ is compact. If $g\in R$ is a nonzero function such that $g-h\in \widehat{\mathfrak{m}}^{n}$ in $\widehat{R}$ for some $n> \alpha$. Then $v'(g-h)\geq n v'(\widehat{\mathfrak{m}})> v'(h)$ for all $v'\in K'$. It follows that $v(g)=v'(g)=v'(h)$ for all $v'\in K'$ and hence they induce the same topology.
\end{proof}

\vspace{0.3cm}
%%%%%%%%%%%%%%%%%%%%%%%%%%%%%
\section{Functions on valuation spaces}\label{fun-val}

In this section we will discuss various classes of functions on valuation space with an emphasis on the quasi-plurisubharmonic (qpsh for short) functions.

\subsection{Bounded homogeneous functions}

Let $X$ be a scheme and $\mathrm{V}_{X}$ be its valuation space. A valuative function $\varphi$ is \emph{homogeneous} if $\varphi(tv)=t\varphi(v)$ for all $v\in \mathrm{V}_X$ and $t\in \mathbb{R}_+$. A valuative function $\varphi$ is \emph{bounded} if $\sup_{v\in\mathrm{V}_{X}^{\ast}}\frac{|\varphi(v)|}{A(v)}<\infty$. The set of bounded homogeneous functions forms an $\mathbb{R}$-linear space, which can be equipped with the norm $\|\varphi\|=\sup_{v\in\mathrm{V}_{X}^{\ast}}\frac{|\varphi(v)|}{A(v)}$, and is denoted by $\mathrm{BH}(X)$. If $\mathfrak{q}$ is a nonzero ideal on $X$, then we define the $\mathfrak{q}$-norm as $\|\varphi\|_{\mathfrak{q}}=\sup_{v\in\mathrm{V}_{X}^{\ast}}\frac{|\varphi(v)|}{A(v)+v(\mathfrak{q})}$.

We also define
$$\|\varphi\|_{\mathfrak{q}}^{+}:=\sup_{v\in\mathrm{V}_{X}^{\ast}}\frac{\varphi(v)}{A(v)+v(\mathfrak{q})}$$
and
$$\|\varphi\|_{\mathfrak{q}}^{-}:=\sup_{v\in\mathrm{V}_{X}^{\ast}}\frac{-\varphi(v)}{A(v)+v(\mathfrak{q})}.$$ Clearly, $\|\varphi\|_{\mathfrak{q}}^+=\|-\varphi\|_{\mathfrak{q}}^-$ and $\|\cdot\|_{\mathfrak{q}}=\max\{\|\cdot\|_{\mathfrak{q}}^{+},\|\cdot\|_{\mathfrak{q}}^{-} \}$.

\begin{lem}
Given two nonzero ideals $\mathfrak{p}$, $\mathfrak{q}$ on $X$, the $\mathfrak{p}$-norm and the $\mathfrak{q}$-norm are equivalent.
\end{lem}

\begin{proof}
We first assume that $\mathfrak{p}=\mathcal{O}_{X}$. Then we have the inequality $\|\cdot \|_{\mathfrak{q}}\leq \|\cdot \| \leq (1+\sup_{v\in\mathrm{V}_{X}^{\ast}}\frac{v(\mathfrak{q})}{A(v)})\|\cdot \|_{\mathfrak{q}}$. Note that $\sup_{v\in\mathrm{V}_{X}^{\ast}}\frac{v(\mathfrak{q})}{A(v)}=\max_{D_i}\frac{\mathrm{ord}_{D_i}(\mathfrak{q})}{A(\mathrm{ord}_{D_i})}<\infty$ where $D_{i}$ runs over all irreducible components of $D$ on a birational model $Y$ such that $(Y,D)$ is a log resolution of $\mathfrak{q}$. This implies that $1+\sup_{v\in\mathrm{V}_{X}^{\ast}}\frac{v(\mathfrak{q})}{A(v)} < \infty$ and leads to the conclusion.
\end{proof}

\begin{prop}\label{norm_reg}
Given a scheme $X$, $\mathrm{BH}(X)$ is a Banach space. If $f:X'\longrightarrow X$ is a regular morphism and $f_{\ast}: \mathrm{V}_{X'}\longrightarrow \mathrm{V}_{X}$ is the induced map, then the induced map $f^{\ast}:\mathrm{BH}(X)\longrightarrow\mathrm{BH}(X')$ by taking $\varphi$ to $\varphi\circ f_{\ast}$ is a bounded linear operator of Banach spaces. More precisely, $\|f^{\ast}(\varphi)\|_{\mathfrak{q}\cdot \mathcal{O}_{X'}}\leq \|\varphi\|_{\mathfrak{q}}$ for any nonzero ideal $\mathfrak{q}$ on $X$.
\end{prop}

\begin{proof}
Note that a bounded homogeneous function $\varphi$ is also a function on $\Lambda_{X}:=\{v\in \mathrm{V}_X^\ast|A(v)=1\}$ with the norm $\sup_{v\in \Lambda_{X}}|\varphi(v)|<\infty$. If $\{\varphi_{m}\}$ is a Cauchy sequence in $\mathrm{BH}(X)$, then $\varphi_{m}$ converges pointwisely to a homogeneous function $\varphi$. Since $\sup_{v\in \Lambda_{X}}|\varphi(v)| \leq \sup_{v\in \Lambda_{X}} |\varphi(v) - \varphi_{m}(v)| + \sup_{v\in \Lambda_{X}} |\varphi_{m}(v)| <\infty$, $\varphi$ is a bounded homogeneous function. This proves that $\mathrm{BH}(X)$ is a Banach space. For the second statement, simply note that $$\|f^{\ast}(\varphi)\|_{\mathfrak{q}\cdot \mathcal{O}_{X'}}=\sup_{v'\in\mathrm{V}_{X'}^{\ast}} \frac{|\varphi(v)|}{A(v')+v'(\mathfrak{q}\cdot \mathcal{O}_{X'})}\leq \sup_{v\in \mathrm{V}_{X}^{\ast}} \frac{|\varphi(v)|}{A(v)+v(\mathfrak{q})}=\|\varphi\|_{\mathfrak{q}}
$$
by Lemma \ref{disc_reg}.
\end{proof}

\begin{rem}
Let $\varphi$ be a bounded homogeneous function such that $\varphi(v)=-v(\mathfrak{a})$ for some nonzero ideal $\mathfrak{a}$ on $X$. It is easy to see that the norm $\|\varphi\|_{\mathfrak{q}}$ is exactly the Arnold multiplicity $\mathrm{Arn}^{\mathfrak{q}}\mathfrak{a}$, and its reciprocal is the log canonical threshold $\mathrm{lct}^{\mathfrak{q}}\mathfrak{a}$. We will discuss this type of functions in detail later.
\end{rem}

\begin{defn}\label{defn:val_fun}
A bounded homogeneous function $\varphi$ is said to be \emph{weakly continuous} if $\varphi$ is continuous on every dual complex $\Delta(Y,D)$.
\end{defn}

\begin{exa}
(1). As we already mentioned, the log discrepancy $A$ is a weakly bounded homogeneous function. \\
(2). If $\{\varphi_k\}$ is a sequence of continuous bounded homogeneous functions which converges to a function $\varphi$ strongly in the norm, then $\varphi$ is weakly continuous.
\end{exa}

\subsection{Ideal functions and qpsh functions}

Given a nonzero ideal $\mathfrak{a}$, we define $|\mathfrak{a}|(v)=-e^{v(\mathfrak{a})}$ by convention. It is obvious that $\log|\mathfrak{a}|$ is a continuous bounded homogeneous function.

\begin{defn}
A bounded homogeneous function $\varphi$ is said to be an \emph{ideal function} if there exists a finite number of nonzero ideals $\mathfrak{a}_j$ and positive real numbers $c_j$ such that $\varphi=\sum_{j=1}^l c_j\log|\mathfrak{a}_j|$.
\end{defn}

\begin{lem}\label{lem:c_norms}
Let $\varphi=\sum_{j=1}^l c_j\log|\mathfrak{a}_j|$ be an ideal function on $X$ and $\mathfrak{q}$ be a nonzero ideal. Then,
$$
\|\varphi\|_{\mathfrak{q}}= \max_E\{\frac{\sum_{j=1}^l c_j \mathrm{ord}_E \mathfrak{a}_j}{A(\mathrm{ord}_E)+\mathrm{ord}_E\mathfrak{q}}\}
$$
for some prime divisor $E$ over $X$.
\end{lem}

\begin{proof}
Let $(Y,D)$ be a log resolution of $\mathfrak{q}\cdot (\prod_{j=1}^l \mathfrak{a}_j)$, and let $D_i$'s be the irreducible components of $D$. By an easy computation, we see that
$$
\|\varphi\|_{\mathfrak{q}}= \max_{D_i}\{\frac{\sum_{j=1}^l c_j \mathrm{ord}_{D_i} \mathfrak{a}_j}{A(\mathrm{ord}_{D_i})+\mathrm{ord}_{D_i}\mathfrak{q}}\}
$$
where $D_i$ runs over all irreducible components of $D$.
\end{proof}

\begin{lem}\label{lem:linear}
Let $\varphi$ be a bounded homogeneous function is determined on some dual complex $\Delta(Y,D)$ in the sense of $\varphi=\varphi \circ r_{Y,D}$. Assume that $\varphi$ is affine on each face of the dual complex $\Delta(Y,D)$ and that $(Y',D') \succeq (Y,D)$. Then $\varphi=\varphi \circ r_{Y',D'}$ which is also affine on each face of the dual complex $\Delta(Y',D')$.
\end{lem}

\begin{proof}
The assumption that $\varphi$ is affine on each face of the dual complex $\Delta(Y,D)$ is equivalent to that $\varphi$ is linear on each simplicial cone of $\mathrm{QM}(Y,D)$. The conclusion follows from the fact that $r_{Y,D}$ is linear on each simplicial cone of $\mathrm{QM}(Y',D')$.
\end{proof}

\begin{defn}
A bounded homogeneous function $\varphi$ is said to be a \emph{quasi-plurisubharmonic} (\emph{qpsh} for short) function if there exists a sequence of ideal functions which converges to $\varphi$ strongly in the norm. The set of qpsh functions, which is a closed convex cone in $\mathrm{BH}(X)$, is denoted by $\mathrm{QPSH}(X)$.
\end{defn}

\begin{defn}
The \emph{support} of a qpsh function is defined to be the set $\{x\in X|x=c_X(v)$ for some nontrivial tempered valuation $v$ such that $\varphi(v)<0\}$.
\end{defn}

If $\varphi=\sum_{j=1}^l c_i\log|\mathfrak{a}_i|$ is an ideal function, then the support of $\varphi$ is the union of the vanishing loci $V(\mathfrak{a}_j)$ and hence proper closed. We will see that the support of is a qpsh function is a countable union of proper closed subsets. See Corollary \ref{cor:supp}.

\begin{prop}\label{qpsh_usc}
Let $\varphi \in \mathrm{QPSH}(X)$ be a qpsh function. Then, $\varphi$ is convex on each face of every dual complex $\Delta(Y,D)$. Moreover, $\varphi\circ r_{Y,D}$ form a decreasing net of continuous functions which converges to $\varphi$ strongly in the norm. In particular, $\varphi$ is weakly continuous and upper-semicontinuous (usc for short).
\end{prop}

\begin{proof}
To show that $\varphi$ is convex on each face of every dual complex $\Delta(Y,D)$, it suffices to prove this when $\varphi$ is an ideal function. We can assume that $\varphi=c\log|\mathfrak{a}|$. Let $\eta$ be a generic point of the intersection of $D_1,\ldots,D_l$. We will prove that $\varphi$ is convex on $\mathrm{QM}_\eta (Y,D)$ which essentially implies the convexity on $\Delta(Y,D)$. To this end, assume that $v=\sum_{j=1}^k \lambda_j v_j$ such that $v$, $v_j \in \mathrm{QM}_\eta (Y,D)$, $\lambda_j>0$ for every $j$ and $\sum_{j=1}^k \lambda_j=1$. Assume further that $\mathfrak{a}\cdot \mathcal{O}_Y$ is principle near $\eta$ generated by $f$. If we consider the local coordinates $\underline{y}=\{y_1,\ldots,y_l\}$ with the origin $\eta$, then $v$ and $v_j$ can be represented by $\alpha=(\alpha^1,\ldots,\alpha^l)$ and $\alpha_j=(\alpha_j^1,\ldots,\alpha_j^l)$ with the values $v(f)=\min \{<\alpha,\beta>|f=\sum c_\beta y^\beta\}$ and $v_j(f)=\min \{<\alpha_j,\beta>|f=\sum c_\beta y^\beta\}$. Obviously, $v(f) \geq \sum_{j=1}^k \lambda_j v_j(f)$ and we obtain the required convexity. If $\mathfrak{a}\cdot \mathcal{O}_Y$ is not principle, then $\varphi$ is the maximum of a finite number of convex functions and hence convex.

Given an arbitrary qpsh function $\varphi$, the functions $\varphi\circ r_{Y,D}$ form a decreasing net because $v \leq r_{Y,D}(v)$, and $\varphi$ is continuous on $\Delta(Y,D)$ because it is the uniform limit function of continuous functions. It suffices to show that $\varphi\circ r_{Y,D}$ converges to $\varphi$ strongly in the norm. To this end, consider a sequence of ideal functions $\varphi_j=c_j \log|\mathfrak{a}_j|$ which converges to $\varphi$ strongly in the norm. We then obtain that
$$
\|\varphi -\varphi\circ r_{Y,D}\|\leq \|\varphi - \varphi_j\|+\|\varphi_j - \varphi\circ r_{Y,D}\| \leq 2 \|\varphi - \varphi_j\|
$$
if $(Y,D)$ is a log resolution of $\mathfrak{a}_j$ which completes the proof.
\end{proof}

\begin{rem}
The proposition above implies that a qpsh function is uniquely determined by its values on divisorial valuations. In fact, if $\varphi$ and $\phi$ have the same values on divisorial valuations, then $\varphi=\phi$ on every dual complex $\Delta(Y,D)$ by the continuity and hence $\varphi=\phi$ on $\mathrm{V}_X$.
\end{rem}

The following example shows that the pointwise limit of a decreasing sequence of ideal functions is not qpsh in general.

\begin{exa}\label{exa:non-psh}
Let $X=\mathrm{Spec }k[x]$ be an affine line, and let $\phi_k=\sum_{j=1}^k \log|f_j|$ where $f_j=x-j$. We see that $\phi_k$ is a decreasing sequence of ideal functions and the pointwise limit function $\varphi$ exists. But $\varphi$ is not qpsh because $\|\varphi -\phi \|\geq 1$ for any ideal function $\phi$.
\end{exa}

If $f:X' \rightarrow X$ is a regular morphism and $\varphi$ is a qpsh function on $X$, then $f^\ast \varphi$ is a qpsh function on $\mathrm{V}_{X'}$ by Proposition \ref{norm_reg}. In particular if $f:U \rightarrow X$ is an open inclusion (resp. $f: \mathrm{Spec} \mathcal{O}_{X,\xi} \rightarrow X$), we say that $f^\ast \varphi$ is the restriction (resp. germ) of $\varphi$, and denote by $\varphi|_U$ (resp. $\varphi_\xi$). Also, restrictions to the neighborhoods of a point $\xi$ induce a map $\mathrm{QPSH}(X) \rightarrow \lim\limits_{\overrightarrow{U \ni \xi}} \mathrm{QPSH}(U)$, and the image of $\varphi$ is also said to be the germ of $\varphi$, and denoted by $\varphi|_\xi$.

If $\xi$ is not contained in the support of a qpsh function $\varphi$, then $\varphi_\xi=0$ by Proposition \ref{qpsh_usc}. However, the following example shows that it could happen that the germ of $\varphi$ is nonzero in the set $\lim\limits_{\overrightarrow{U \ni \xi}} \mathrm{QPSH}(U)$.

\begin{exa}
Let $X=\mathrm{Spec}k[x]$ be an affine line, and let $\phi_k=\sum_{m=1}^k \frac{1}{2^m}\log|f_m|$ where $f_m=x-\frac{1}{m}$. It is easy to see that $\phi_k$ converges to a function $\phi$ strongly in the norm. Note that the origin is not contained in the support of $\phi$, but the germ of $\phi$ in $\lim\limits_{\overrightarrow{U \ni 0}} \mathrm{QPSH}(U)$ is nonzero.
\end{exa}

From the above example we see that if we define $\|\varphi|_\xi \|:= \inf_{U \ni \xi} \|\varphi|_U\|$, then $\|\cdot \|$ is only a semi-norm.

\begin{prop}\label{prop:local}
There is a surjective map of convex cones
\begin{align*}
r: \lim\limits_{\overrightarrow{U \ni \xi}} \mathrm{QPSH}(U)\longrightarrow \mathrm{QPSH}(\mathrm{Spec} \mathcal{O}_{X,\xi})
\end{align*}
which preserves the semi-norm, and also preserves $\|\cdot\|^+$ and $\|\cdot\|^-$.
\end{prop}

\begin{proof}
If $\varphi=c\log|\mathfrak{a}|$ and $\varphi'=c'\log|\mathfrak{a}|$, then we claim that $\|\varphi|_\xi -\varphi'|_\xi\|=\|\varphi_\xi -\varphi'_\xi\|$. To this end, let $\mu:(Y,D)\rightarrow X$ be a log resolution of $\mathfrak{a}\cdot \mathfrak{a}'$, and let $\mathfrak{a}\cdot \mathcal{O}_Y=\mathcal{O}_Y(-F)$ and $\mathfrak{a}'\cdot \mathcal{O}_Y=\mathcal{O}_Y(-F')$. One can easily check that
$$
\|\varphi|_\xi -\varphi'|_\xi\|= \max_{D_i \in \mathcal{S}} \frac{|\mathrm{ord}_{D_i}F-\mathrm{ord}_{D_i}F'|}{A(\mathrm{ord}_{D_i})}
$$
where $\mathcal{S}$ consists of irreducible components $D_i$ of $D$ such that $\mu(D_i)$ contains $\xi$ in its support. This implies the claim.

Given a qpsh function $\varphi_\xi \in \mathrm{QPSH}(\mathrm{Spec} \mathcal{O}_{X,\xi})$, there exists a sequence of ideal functions $\varphi_{\xi,k}=c_i\log|\mathfrak{a}_{\xi,i}|$ which converges to $\varphi_\xi$ strongly in the norm. Let $\mathfrak{a}_i$ be ideals on $X$ such that $\mathfrak{a}_i \cdot \mathcal{O}_{X,\xi}=\mathfrak{a}_{\xi,i}$. We have that $\varphi_k=c_i\log|\mathfrak{a}_i|$ converges to a qpsh function in $\lim\limits_{\overrightarrow{U \ni \xi}} \mathrm{QPSH}(U)$ strongly in the norm due to the previous claim. Therefore we obtain the surjectivity of $r$.

Finally, for two qpsh function $\varphi$ and $\varphi'$ on an open neighborhood of $\xi$, the equality $\|\varphi|_\xi -\varphi'|_\xi\|=\|\varphi_\xi -\varphi'_\xi\|$ follows from the claim in the first paragraph. Apply a similar argument to $\|\cdot\|^+$ and $\|\cdot\|^-$, we obtain the conclusion.
\end{proof}

From the discussion above, we see that $\varphi|_\xi$ provides more information while it is not a valuative function. We sometimes identify $\varphi|_\xi$ and $\varphi_\xi$ as the germ of $\varphi$ at $\xi$.

\vspace{0.3cm}
%%%%%%%%%%%%%%%%%%%%%%%%%%%%%
\section{Multiplier ideals}

In this section we will discuss the multiplier ideals of qpsh functions. Recall that a graded sequence of ideals $\mathfrak{a}_{\bullet}$ is a sequence of ideals which satisfies $\mathfrak{a}_{m} \cdot \mathfrak{a}_{n} \subseteq \mathfrak{a}_{m+n}$. By convention we put $\mathfrak{a}_0=\mathcal{O}_X$, and we say $\mathfrak{a}_{\bullet}$ is nontrivial if $\mathfrak{a}_m\neq 0$ for some positive integer $m$. Note that in this case there are infinitely many $m$ such that $\mathfrak{a}_m\neq 0$. A subadditive sequence of ideals $\mathfrak{b}_{\bullet}$ is a one-parameter family $\mathfrak{b}_t$ satisfying $\mathfrak{b}_{s} \cdot \mathfrak{b}_{t} \supseteq\mathfrak{b}_{s+t}$ for every $s$, $t \in \mathbb{R}_+$. Similarly, we put $\mathfrak{b}_0= \mathcal{O}_X$ and we say that $\mathfrak{b}_{\bullet}$ is nontrivial if $\mathfrak{b}_t\neq 0$ for all $t\in \mathbb{R}_+$. Throughout this paper, every sequence of ideals is assumed to be nontrivial. We define $v(\mathfrak{a}_\bullet)=\inf_{m\geq 1} \frac{v(\mathfrak{a}_m)}{m}$ and $v(\mathfrak{b}_\bullet)=\sup_{t>0} \frac{v(\mathfrak{b}_t)}{t}$ as in [\ref{ELMNP}]. We similarly define $|\mathfrak{a}_\bullet|(v)=e^{-v(\mathfrak{a}_\bullet)}$ and $|\mathfrak{b}_\bullet|(v)=e^{-v(\mathfrak{b}_\bullet)}$ for a graded sequence and a subadditive sequence of ideals respectively.

\subsection{Multiplier ideals}

\begin{defn}\label{def_mult}
For a bounded homogeneous function $\varphi\in\mathrm{BH}(X)$, the multiplier ideal $\mathcal{J}(\varphi)$ of $\varphi$ is the largest ideal in the set of nonzero ideals $\{\mathfrak{a}| \| \log|\mathfrak{a}|-\varphi\|^{+}<1 \}$. If the above set is empty, then we define $\mathcal{J}(\varphi)=(0)$.
\end{defn}

\begin{rem}\label{rem_mult}
We will see that the above set is always nonempty when $\varphi$ is qpsh and $\mathcal{J}(\varphi)$ is therefore nonzero (cf. Remark \ref{qpsh_mult_r}). Moreover, we have the inequality $\varphi \leq \log|\mathcal{J}(\varphi)|$ (cf. Remark \ref{qpsh_mult_r}), and hence the inequality $\| \log|\mathcal{J}(\varphi)|-\varphi\|<1$ holds.
\end{rem}

The following proposition shows that the above definition of multiplier ideals coincides with the 'classical definition' of multiplier ideals.

\begin{prop}\label{prop:ideal_mult}
If $\varphi$ is an ideal function and we write $\varphi= \sum_{i=1}^l c_{i}\log|\mathfrak{a}_{i}|$, then $\mathcal{J}(\varphi)=\mathcal{J}(\prod_{i=1}^l \mathfrak{a_{i}}^{c_{i}})$.
\end{prop}

\begin{proof}
Let $\pi:(Y,D)\longrightarrow X$ be a log resolution of $\prod_{i=1}^l \mathfrak{a_{i}}$, and $\mathfrak{a}_{i}\cdot \mathcal{O}_{Y}=\mathcal{O}_{Y}(-F_{i})$ with $F_{i}$ being supported in $D$. Since $\mathcal{J}(\prod_{i=1}^l \mathfrak{a_{i}}^{c_{i}})=\pi_{\ast}\mathcal{O}_{Y}(K_{Y/X}-\llcorner \sum_{i=1}^l c_{i}F_{i} \lrcorner)$, it is easy to check that $\| \log |\mathcal{J}(\prod_{i=1}^l \mathfrak{a_{i}}^{c_{i}})|-\varphi \|^{+}<1$ which immediately implies that $\mathcal{J}(\varphi)\supseteq\mathcal{J}(\prod_{i=1}^l \mathfrak{a_{i}}^{c_{i}})$.

Conversely, if $f\in \Gamma(U,\mathcal{J}(\varphi))$ is a regular function on an affine open subset $U$, then $\| \log|f|-\varphi|_U \|^{+}<1$. It follows that $v(f)+A(v) +\varphi(f)>0$ for all nontrivial tempered valuations $v$ on $U$. In particular, $\mathrm{ord}_{E}f+ \mathrm{ord}_{E}K_{Y/X} +1 > -\varphi(\mathrm{ord}_{E}) $ for any prime divisor $E$ on $\pi^{-1}U$. Thus $f\in \Gamma(U, \mathcal{J}(\prod_{i=1}^l \mathfrak{a_{i}}^{c_{i}}))$ and it follows that $\mathcal{J}(\varphi)\subseteq\mathcal{J}(\prod_{i=1}^l \mathfrak{a_{i}}^{c_{i}})$.
\end{proof}

The lemmas below will be frequently used in this paper.

\begin{lem}\label{norm_1}
Given a nonzero ideal $\mathfrak{q}$ and a qpsh function $\varphi\in\mathrm{QPSH}(X)$, $\mathfrak{q}\subseteq \mathcal{J}(\lambda \varphi)$ if and only if $ \lambda^{-1}> \|\varphi\|_{\mathfrak{q}}$. Thus $\|\varphi\|_{\mathfrak{q}}^{-1}=\min\{t| \mathfrak{q}\nsubseteq \mathcal{J}(t\varphi)\}$
\end{lem}

\begin{proof}
If $\mathfrak{q}\subseteq \mathcal{J}(\lambda \varphi)$, then $\| \log |\mathfrak{q}|-\lambda \varphi\|^{+}<1$ by definition. That is, $\sup\limits_{v\in\mathrm{V}_{X}^{\ast}}\frac{-v(\mathfrak{q})-\lambda \varphi(v)}{A(v)}<1$. This implies that $-v(\mathfrak{q})-\lambda \varphi(v) \leq (1-\varepsilon)A(v)$ for every $v\in\mathrm{V}_{X}^{\ast}$. Thus $\frac{-\lambda \varphi(v)}{A(v)+v(\mathfrak{q})}\leq \frac{(1-\varepsilon)A(v)+v(\mathfrak{q})}{A(v)+v(\mathfrak{q})}\leq (1-\varepsilon)+\varepsilon \|\log|\mathfrak{q}|\|_{\mathfrak{q}}<1$ by Lemma \ref{lem:c_norms}. We obtain $ \lambda^{-1}> \|\varphi\|_{\mathfrak{q}}$ by definition.

Conversely we assume that $\|\varphi\|_\mathfrak{q}=\sup\limits_{v\in\mathrm{V}_{X}^{\ast}}\frac{-\lambda \varphi(v)}{A(v)+v(\mathfrak{q})}<1$. Then $-\lambda\varphi(v)\leq (1-\varepsilon)(A(v)+v(\mathfrak{q}))$. Therefore $\frac{-v(\mathfrak{q})-\lambda \varphi(v)}{A(v)}\leq 1-\varepsilon- \varepsilon\frac{v(\mathfrak{q})}{A(v)}\leq 1-\varepsilon$ for a sufficiently small $\varepsilon$ which leads to the conclusion $\mathfrak{q}\subseteq \mathcal{J}(\lambda \varphi)$.
\end{proof}

\begin{lem}\label{norm_2}
Let $\xi$ be a point of a scheme $X$, and $\varphi$ be a qpsh function. Assume that the multiplier ideal $\mathcal{J}(\varphi)$ is nonzero. (In fact, this assumption automatically holds by Lemma \ref{qpsh_mult} and Remark \ref{qpsh_mult_r}). Then, \\
(1). $\mathcal{J}(\varphi|_{U})=\mathcal{J}(\varphi)\cdot \mathcal{O}_{U}$.\\
(2). $\mathcal{J}(\varphi_{\xi})=\mathcal{J}(\varphi)\cdot \mathcal{O}_{X,\xi}$.\\
(3). Set $\lambda^{-1}:=\|\varphi\|_{\mathfrak{q}}$. If $\xi \in \mathrm{V}(\mathcal{J}(\lambda \varphi):\mathfrak{q})$, then $\|\varphi\|_{\mathfrak{q}}=\|\varphi_{\xi}\|_{\mathfrak{q}\cdot \mathcal{O}_{X,\xi}}$.
\end{lem}

\begin{proof}
(1). Since $\|\log|\mathcal{J}(\varphi)\cdot \mathcal{O}_{U}|-\varphi|_{U}\|^{+} \leq \|\log|\mathcal{J}(\varphi)|-\varphi \|^+ <1$, we have $\mathcal{J}(\varphi)\cdot \mathcal{O}_{U} \subseteq \mathcal{J}(\varphi|_{U})$. On the other hand, if we denote by $\mathfrak{m}$ the defining ideal of $X\setminus U$, then there exists a sufficiently large integer $k$ such that $v(\mathcal{J}(\varphi))\leq v(\mathfrak{m}^k)$ for all valuations $v$ centered in $X\setminus U$. Now we extend $\mathcal{J}(\varphi|_{U})$ to $X$ and still denote it by $\mathcal{J}(\varphi|_{U})$. Therefore $\|\log|\mathcal{J}(\varphi|_{U})\cdot \mathfrak{m}^{k}|-\varphi\|^{+}<1$ which implies $\mathcal{J}(\varphi|_{U})\subseteq \mathcal{J}(\varphi)\cdot \mathcal{O}_{U}$.

(2). First note that $\|\log|\mathcal{J}(\varphi)\cdot \mathcal{O}_{X,\xi}|-\varphi_{\xi}\|^{+} \leq \|\log|\mathcal{J}(\varphi)|-\varphi \|^+ <1$, and it follows that $\mathcal{J}(\varphi)\cdot \mathcal{O}_{X,\xi} \subseteq \mathcal{J}(\varphi_{\xi})$. For the inverse inclusion, we see that if $f\in \mathcal{J}(\varphi_{\xi})$, then there exists an open neighborhood $U$ of $\xi$ such that $\|\log|f|-\varphi|_{U}\|^{+}<1$ by Proposition \ref{prop:local}. Thus $f\in \mathcal{J}(\varphi|_{U})\cdot \mathcal{O}_{X,\xi}=\mathcal{J}(\varphi)\cdot \mathcal{O}_{X,\xi}$.

(3). It is obvious that $\|\varphi\|_{\mathfrak{q}}\geq\|\varphi_{\xi}\|_{\mathfrak{q}\cdot \mathcal{O}_{X,\xi}}$ by Proposition \ref{norm_reg}. If $\xi \in \mathrm{V}(\mathcal{J}(\lambda \varphi):\mathfrak{q})$, then $(\mathcal{J}(\lambda \varphi_{\xi}):\mathfrak{q}\cdot \mathcal{O}_{X,\xi})=(\mathcal{J}(\lambda \varphi):\mathfrak{q})\cdot \mathcal{O}_{X,\xi} \neq \mathcal{O}_{X,\xi}$. Therefore $\mathfrak{q}\cdot \mathcal{O}_{X,\xi} \nsubseteq \mathcal{J}(\lambda \varphi|_{\xi})$ and $\lambda^{-1}\leq \|\varphi_{\xi}\|_{\mathfrak{q}\cdot \mathcal{O}_{X,\xi}}$ by Lemma \ref{norm_1}.
\end{proof}

\subsection{Algebraic qpsh functions}\label{section:alg-qpsh}

\begin{defn}
A qpsh function $\varphi \in\mathrm{QPSH}(X) $ is \emph{algebraic} if it is the limit function of an increasing sequence of ideal functions $\varphi= \lim\limits_{m \rightarrow \infty} \varphi_{m}$ (in the norm). Note that $\varphi$ being algebraic implies that $t\varphi$ is algebraic for any $t\in \mathbb{R}_{>0}$, and that $\varphi + \psi$ is algebraic provided $\psi$ is another algebraic qpsh function. Thus the set of algebraic qpsh functions is a convex subcone of $\mathrm{QPSH}(X)$, and denoted by $\mathrm{QPSH}^a(X)$.
\end{defn}

 An algebraic function is lower-semicontinuous (lsc) on $\mathrm{V}_X$ by its definition, and it is usc by Proposition \ref{qpsh_usc}, so it is continuous. We will see that in the above definition the phrase 'in the norm' is not necessary, that is, the pointwise limit of an increasing sequence of ideal functions is algebraic qpsh (cf. Lemma \ref{cl_sup}). One can compare this fact with Remark \ref{qpsh_ideal}. The following example shows that a qpsh function is not necessarily algebraic.

\begin{exa}\label{non_alg_qpsh}
Let $X=\mathrm{Spec} k[x_1,x_2]$ be the affine plane. If we set $\phi_{k}=\sum\limits_{l=1}^{k} \frac{1}{2^{l}}\log |f_{l}|$ where $f_{l}=x_{1}+x_{2}^{2^{l}}$, then $\phi_{k}$ converges to a qpsh function $\phi$ strongly in the norm. However, the qpsh function $\phi$ is not algebraic since there is no ideal function $\varphi \leq \phi$.
\end{exa}

The following lemma shows that a graded system of ideals naturally induces an algebraic qpsh functions.

\begin{lem}\label{grad_alg}
Let $\mathfrak{a}_{\bullet}$ be a graded sequence of ideals. If we define $\log|\mathfrak{a}_\bullet|(v)=-v(\mathfrak{a}_\bullet)$, then $\log|\mathfrak{a}_\bullet|$ is an algebraic qpsh function.
\end{lem}

\begin{proof}
It suffices to show that there exists a subsequence of $\{\mathfrak{a}_m\}$ such that $\{\varphi_k:= \frac{1}{m_k}\log|\mathfrak{a}_{m_k}|\}$ is an increasing sequence of ideal functions which converges to a qpsh function strongly in the norm. Let $\mathfrak{b}_{\bullet}$ be a sequence of ideals such that $\mathfrak{b}_{t}=\mathcal{J}(\mathfrak{a}_{\bullet}^t)$ (cf. [\ref{Laz}]). Note that $\mathfrak{b}_{\bullet}$ is subadditive of controlled growth (cf. [\ref{JM}, Section 2, Section 6, Appendix]). Now we fix an integer $m$ such that $\mathfrak{a}_m \neq 0$. Since $\mathfrak{b}_{m}\supseteq \mathcal{J}(\mathfrak{a}_{m})\supseteq \mathfrak{a}_{m}$, we have $v(\mathfrak{b}_{m})\leq v(\mathfrak{a}_{m})$. Since $v(\mathfrak{b}_{m})+A(v)-\frac{1}{k}v(\mathfrak{a}_{mk})>0$ for all nontrivial tempered valuations $v$ where $k$ is a sufficiently divisible integer, we have $\frac{v(\mathfrak{a}_{mk})}{mk}< \frac{v(\mathfrak{b}_{m})}{m}+\frac{A(v)}{m}$. From the inequality $\frac{v(\mathfrak{b}_{m})}{m}\leq \frac{v(\mathfrak{a}_{mk})}{mk}< \frac{v(\mathfrak{b}_{m})}{m}+\frac{A(v)}{m}$, we have that $\|\frac{1}{mk} \log|\mathfrak{a}_{mk}|-\frac{1}{mkl} \log|\mathfrak{a}_{mkl}|\|<\frac{1}{m}$ for every positive integer $l$. As we multiply $m$, we obtain the desired sequence of ideal functions.
\end{proof}

\begin{defn}
Let $\varphi \in\mathrm{BH}(X)$ be a bounded homogeneous function. Its envelope ideal $\mathfrak{a}(\varphi)$ is the largest ideal in the set $\{\mathfrak{a}|\log|\mathfrak{a}| \leq \varphi\}$ if this set is nonempty. If it is empty, we set $\mathfrak{a}(\varphi)=0$.
\end{defn}

\begin{prop}
If $\varphi$ is qpsh and $\mathfrak{a}(\varphi)$ is nonzero, then the envelope ideal can be written explicitly as $\Gamma(U,\mathfrak{a}(\varphi)):=\{f\in \mathcal{O}_X(U)|v(f)+\varphi(v)\geq 0$ for every $v\in \mathrm{V}_{U}^\ast\}$ on every open subset $U$.
\end{prop}

\begin{proof}
Since the question is local, we can assume that $X=\mathrm{Spec} A$ is affine. It suffices to prove that the ideal $\mathfrak{a}$, defined by $\mathfrak{a}(U):=\{f\in \mathcal{O}_X(U)|v(f)+\varphi(v)\geq 0$ for every $v\in \mathrm{V}_{U}^\ast\}$ on every open subset $U$, is coherent. To this end, we write $I:=\mathfrak{a}(X)$, and we will prove that $\mathfrak{a}(U_g)=I_g$ for any nonzero regular function $g \in A$, where $U_g$ denotes the affine open subset defined by $g$. Since $\mathfrak{a}(U_g)\supseteq I_g$  by definition, we only need to prove the converse inclusion. Note that there exists a large integer $k$ such that $k v(g) \geq v(\mathfrak{a}(\varphi))$ for every nontrivial tempered valuation $v$ centered in the locus $V(g)$, and hence $k\log|g|(v) \leq \varphi(v)$. If $f$ is a regular function on $U_g$ such that $v(f)+\varphi(v) \geq 0$ for every $v\in \mathrm{V}_{U_g}^\ast$, then $v(fg^k)+\varphi(v) \geq 0$ for every $v \in \mathrm{V}_X^\ast$ which implies that $f \in I_g$.
\end{proof}

If we set $\mathfrak{a}(\varphi)_m=\mathfrak{a}(m \varphi)$, then $\{\mathfrak{a}(\varphi)_\bullet\}$ is a (possibly trivial) graded sequence of ideals. The following lemma shows that every algebraic qpsh function is of the form $\log|\mathfrak{a}_\bullet|$.

\begin{lem}\label{alg_grad}
If $\varphi \in\mathrm{QPSH}^a(X)$ is an algebraic qpsh function, then $\varphi=\log|\mathfrak{a}(\varphi)_\bullet|$.
\end{lem}

\begin{proof}
Given an arbitrary small positive number $\varepsilon$, there exist an ideal $\mathfrak{a}$ on $X$ and an integer $m$ such that $\frac{1}{m}\log|\mathfrak{a}|\leq \varphi$ and $\|\frac{1}{m}\log|\mathfrak{a}|-\varphi\|<\varepsilon$. We have $\frac{1}{m}\log|\mathfrak{a}(\varphi)_m|\geq \frac{1}{m}\log|\mathfrak{a}|$ by definition and the conclusion follows.
\end{proof}

By combining Lemma \ref{prop:ideal_mult}, Lemma \ref{grad_alg} and Lemma \ref{alg_grad}, we see that a bounded homogeneous function is algebraic qpsh if and only if it is derived from a graded sequence of ideals. Readers can compare the following theorem with Theorem \ref{ch_qpsh}.

\begin{thm}\label{ch_alg}
Let $\varphi$ be a bounded homogeneous function. Then the followings are equivalent. \\
(1). $\varphi \in \mathrm{QPSH}^a(X)$ is algebraic qpsh.\\
(2). There exists a graded sequence of ideals $\mathfrak{a}_\bullet$ such that $\varphi=\log|\mathfrak{a}_\bullet|$. \\
(3). The graded system of ideals $\mathfrak{a}(\varphi)_\bullet$ is nontrivial and $\varphi=\log|\mathfrak{a}(\varphi)_\bullet|$.
\end{thm}

\begin{proof}
If we assume (1), then (3) holds by Lemma \ref{alg_grad}. Note that (3) implies (2) if we simply put $\mathfrak{a}_\bullet=\mathfrak{a}(\varphi)_\bullet$. Finally, (1) follows from (2) by Lemma \ref{grad_alg}.
\end{proof}

We will use the following easy lemma. For the convenience of readers we present a proof here.

\begin{lem}\label{alg_mult}
Let $\varphi \in\mathrm{QPSH}^a(X)$ be an algebraic qpsh function. \\
(1). Assume that $\{\varphi_{m}\}$ is an increasing sequence of qpsh functions which converges to $\varphi$ strongly in the norm. Then $\mathcal{J}(\varphi)=\mathcal{J}(\varphi_{m})$ for $m$ sufficiently large.\\
(2). Assume that $f:X'\longrightarrow X$ is a regular morphism of schemes. Then $f^{\ast}\varphi$ is algebraic qpsh.
\end{lem}

\begin{proof}
(1). We see that $\|\log|\mathcal{J}(\varphi)|-\varphi\|^{+}=1-\varepsilon$ for some positive number $\varepsilon$. If $\|\varphi-\varphi_{m}\|<\varepsilon$, then $\|\log|\mathcal{J}(\varphi)|-\varphi_m\|^{+}<1$ and $\mathcal{J}(\varphi) \subseteq \mathcal{J}(\varphi_{m})$. The inverse inclusion $\mathcal{J}(\varphi) \supseteq \mathcal{J}(\varphi_{m})$ is obvious because $\varphi\geq \varphi_{m}$.\\
(2). Assume $\varphi_{m}$ is an increasing sequence of ideal functions which converges to $\varphi$ strongly in the norm. Then $f^{\ast}\varphi_{m}$ is also an increasing sequence of ideal functions which converges to $f^{\ast}\varphi$ strongly in the norm by Proposition \ref{norm_reg}. This implies that $f^{\ast}\varphi$ is algebraic qpsh.
\end{proof}

By combining Lemma \ref{grad_alg} and Proposition \ref{alg_mult}(1), we see that the definition of valuative multiplier ideals of algebraic functions coincides with the 'classical definition' of asymptotic multiplier ideals.

\begin{cor}\label{cor:alg-mult-ideal}
Let $\mathfrak{a}_{\bullet}$ be a graded sequence of ideals. If we write $\varphi=\log|\mathfrak{a}_\bullet|$, then $\mathcal{J}(\varphi)=\mathcal{J}(\mathfrak{a}_\bullet)$.
\end{cor}

\subsection{General qpsh functions}

\begin{lem}\label{cl_sup}
If $\{\varphi_{\lambda}\}$ is a family of (algebraic) qpsh functions, then $\sup_\lambda \varphi_{\lambda}$ is an (algebraic) qpsh function. Therefore, the convex cone $\mathrm{QPSH}(X)$ (resp. $\mathrm{QPSH}^a(X)$) is closed under taking the supremum.
\end{lem}

\begin{proof}
We firstly assume that $\{\varphi_{\lambda}\}$ is a family of algebraic qpsh functions, and we write $\psi=\sup_{\lambda} \varphi_{\lambda}$. Since $\psi \geq \varphi_{\lambda}$ for every $\lambda$, $\mathfrak{a}(\psi)_m \supseteq \mathfrak{a}(\varphi_\lambda)_m$. It follows that $\log|\mathfrak{a}(\psi)_\bullet| \geq \log|\mathfrak{a}(\varphi_\lambda)_\bullet|=\varphi_\lambda$. Therefore $\psi=\log|\mathfrak{a}(\psi)_\bullet|$ is algebraic qpsh.

Now we treat the case when $\{\varphi_{\lambda}\}$ is a family of general qpsh functions. For each $\lambda$, we assume that $\{\varphi_{\lambda,m}\}$ is a sequence of ideal functions which converges to $\varphi_{\lambda}$ strongly in the norm such that $\|\varphi_{\lambda}-\varphi_{\lambda,m}\|<\frac{1}{m}$. If we set $\psi_{m}=\sup_{\lambda}\varphi_{\lambda,m}$ which is algebraic qpsh by the previous argument, then $\|\psi-\psi_m\| \leq \frac{1}{m}$ and it follows that $\{\psi_{m}\}$ is a sequence which converges to $\psi$ strongly in the norm.
\end{proof}

Since the convex cones $\mathrm{QPSH}(X)$ and $\mathrm{QPSH}^a(X)$ are closed under taking the supremum by Lemma \ref{cl_sup}, we can introduce the following definition.

\begin{defn}
Let $\varphi$ be a bounded homogeneous function. Assume that the set $\{\psi \in \mathrm{QPSH}(X)| \psi \leq \varphi\}$ is nonempty. Then we say the maximal function in this set the qpsh envelope function. We similarly define the algebraic qpsh envelope function of $\varphi$ if it exists.
\end{defn}

In general, we cannot ensure the sets defined as above are nonempty. For instance, the function in Example \ref{exa:non-psh} is bounded homogeneous but its qpsh envelope function does not exist. Also note that the function $\phi$ in Example \ref{non_alg_qpsh} is qpsh itself but its algebraic qpsh envelope function does not exist.

\begin{lem}\label{qpsh_eve}
Let $\varphi$ be a bounded homogeneous function that is determined on some dual complex $\Delta(Y,D)$ in the sense of $\varphi=\varphi \circ r_{Y,D}$. Then, its qpsh envelope function $\psi$ exists. Further, $\psi$ is algebraic qpsh.
\end{lem}

\begin{proof}
Let $Z \subseteq X$ be the image of the reduced divisor $D$ on $X$, and $\mathfrak{m}$ be the defining ideal of $Z$. Since $\log|\mathfrak{m}|$ is strictly negative on $\Delta(Y,D)$ and $\varphi$ is bounded on $\Delta(Y,D)$, there exists an integer $k$ such that $k\log|\mathfrak{m}| \leq \varphi$ on $\Delta(Y,D)$. Because $\varphi$ is determined on the dual complex $\Delta(Y,D)$ in the sense of $\varphi=\varphi \circ r_{Y,D}$, we have that $k\log|\mathfrak{m}| \leq \varphi$ on $\mathrm{V}_X$. It follows that its algebraic qpsh envelope function $\phi$ exists. In particular, its qpsh envelope function $\psi$ exists.

Now we will show that $\psi=\phi$. Set $\mu_{1}=\max_{v\in\Delta(Y,D)} |v(\mathfrak{m})|$ and $\mu_{2}=\min_{v\in\Delta(Y,D)} |v(\mathfrak{m})|$. For any small number $\varepsilon>0$, we choose $\delta\ll 1$ such that $(1+\frac{\mu_{1}}{\mu_{2}})\delta< \varepsilon$ and an ideal function $\phi'$ such that $\|\phi'-\psi\|<\delta$. Note that for every valuation $v\in\Delta(Y,D)$ we have
$$
\psi (v)>\phi'(v)-\frac{\delta}{\mu_{2}}v(\mathfrak{m})\geq \phi'(v)-\frac{\delta \mu_{1}}{\mu_{2}}> \psi (v)-(1+\frac{\mu_{1}}{\mu_{2}})\delta.
$$
After replacing $\phi'$ by $\phi'+ \frac{\delta}{\mu_2}\log|\mathfrak{m}|$, we can assume that $\phi' \leq \psi$ and $|\psi(v)-\phi'(v)|< \varepsilon$ on $\Delta(Y,D)$. Because $\varphi=\varphi \circ r_{Y,D}$, we obtain that $\phi' \leq \varphi$. It follows that $\phi' \leq \phi \leq \psi$ by the definition of the qpsh envelope function. Since $\varepsilon$ can be chosen arbitrary small, this forces $\phi=\psi$ on $\Delta(Y,D)$. If we pick any higher log resolution $(Y',D')$, we can show that $\phi=\psi$ on $\Delta(Y',D')$ by the same argument. The conclusion therefore follows from Proposition \ref{qpsh_usc}.
\end{proof}

The above lemma leads to the following definition.

\begin{defn}
Let $\varphi \in\mathrm{QPSH}(X)$ be a qpsh function. We denote by $\varphi_{Y,D}$ the qpsh envelop function of $\varphi \circ r_{Y,D}$.
\end{defn}

\begin{lem}\label{qpsh_alg}
Let $\varphi$ be a qpsh function. Then there exists a decreasing sequence of algebraic functions which converges to $\varphi$ strongly in the norm.
\end{lem}

\begin{proof}
Let $\{\varphi_m\}$ be a sequence of ideal functions which converges to $\varphi$ strongly in the norm. We can assume that $\varphi_m=c_m\log|\mathfrak{a}_m|$ and $\|\varphi -\varphi_m\|< \frac{1}{m}$. Let $(Y,D)$ be a log resolution $\mathfrak{a}_1$. It is easy to see that $\|\varphi\circ r_{Y,D}-\varphi_1\|<1$, and therefore $\|\varphi_{Y,D}-\varphi_1\|<1$. We deduce that $\|\varphi_{Y,D}-\varphi\|<2$. Now we replace $\varphi_1$ by $\varphi_{Y,D}$ and continue this process. Note that if $(Y',D')\succeq (Y,D)$, then $\varphi_{Y',D'}\leq \varphi_{Y,D}$ by Proposition \ref{qpsh_usc}. We easily obtain the required decreasing sequence of algebraic functions.
\end{proof}

\begin{lem}\label{qpsh_mult}
Let $\{\varphi_m\}$ be a sequence of qpsh functions which converges to a qpsh function $\varphi$ strongly in the norm. Then $\mathcal{J}(\varphi)=\mathcal{J}((1+\varepsilon)\varphi_m)$ for a sufficiently small positive real number $\varepsilon>0$ and a sufficiently large integer $m$.
\end{lem}

\begin{proof}
First we prove that $\mathcal{J}(\varphi)\subseteq \mathcal{J}((1+\varepsilon)\varphi_m)$ for a sufficiently small number $\varepsilon>0$ and a sufficiently large integer $m$. To this end, we pick a sufficiently small number $\varepsilon>0$ such that $\mathcal{J}(\varphi)=\mathcal{J}((1+\varepsilon)\varphi)$. Since $\mathcal{J}((1+\varepsilon)\varphi)\subseteq \mathcal{J}((1+\varepsilon)\varphi_m)$ provided that $m$ is sufficiently large, we have $\mathcal{J}(\varphi)\subseteq \mathcal{J}((1+\varepsilon)\varphi_m)$. Conversely, we pick a sufficiently large integer $m$ such that $\|\varphi-\varphi_m\|<1-\frac{1}{1+\varepsilon}$. Applying Lemma \ref{norm_1} again, we see that if $f\in\mathcal{J}((1+\varepsilon)\varphi_m)$, then $\|\varphi_m\|_{f}<\frac{1}{1+\varepsilon}$ and hence $\|\varphi\|_{f}\leq \|\varphi_m\|_{f}+\|\varphi-\varphi_m\|_{f}<1$ which implies that $f\in\mathcal{J}(\varphi)$.
\end{proof}

\begin{rem}\label{qpsh_mult_r}
Note that for an algebraic qpsh function $\phi$, we always have $\phi \leq \log|\mathcal{J}(\phi)|$ by [\ref{JM}, Proposition 6.2 and 6.5]. It follows by Lemma \ref{qpsh_alg} and Lemma \ref{qpsh_mult} that $\mathcal{J}(\varphi)$ is nonzero and $(1+\varepsilon)\varphi \leq (1+\varepsilon)\varphi_m \leq \log|\mathcal{J}(\varphi)|$ where $\{\varphi_m\}$ is a decreasing sequence of algebraic functions which converges to a qpsh function $\varphi$ strongly in the norm. Since $\varepsilon$ can be chosen arbitrary small, we immediately obtain that $\varphi \leq \log|\mathcal{J}(\varphi)|$.
\end{rem}

Now we discuss the multiplier ideals of general qpsh functions.

\begin{prop}\label{subadditivity}
Let $\varphi \in \mathrm{QPSH}(X)$ be a qpsh function on $X$.\\
(1). Assume that $\psi$ is another qpsh function on $X$. Then,
$$
\mathcal{J}(\varphi+\psi)\subseteq\mathcal{J}(\varphi)\cdot\mathcal{J}(\psi).
$$
(2). Assume that $f:X'\longrightarrow X$ is a regular morphism of schemes. Then,
$$
\mathcal{J}(\varphi)\cdot \mathcal{O}_{X'}=\mathcal{J}(f^{\ast}\varphi).
$$
\end{prop}

\begin{proof}
(1). By Lemma \ref{qpsh_alg} we can assume that there are decreasing sequences of algebraic functions $\{\varphi_{m}\}$ and $\{\psi_{m}\}$ convergent to $\varphi$ and $\psi$ strongly in the norm respectively. Then for some sufficiently large integer $m$, by Lemma \ref{qpsh_mult} we have $\mathcal{J}(\varphi+\psi)=\mathcal{J}((1+\varepsilon)(\varphi_{m}+\psi_{m}))\subseteq \mathcal{J}((1+\varepsilon)\varphi_{m})\cdot\mathcal{J}((1+\varepsilon)\psi_{m}) =\mathcal{J}(\varphi)\cdot\mathcal{J}(\psi)$ since $\varphi_{m}+\psi_{m}$ converges to $\varphi+\psi$ strongly in the norm. The inclusion appeared in the above inequality follows from [\ref{JM}, Theorem A.2].\\
(2). Since $f$ is regular, for any ideal function $\phi=\sum_{i} c_{i}\log \mathfrak{a}_{i}$, we have $$
\mathcal{J}(\phi)\cdot \mathcal{O}_{X'}=\mathcal{J}(\prod_{i} \mathfrak{a}_{i}^{c_{i}})\cdot \mathcal{O}_{X'}=\mathcal{J}(\prod_{i} (\mathfrak{a}_{i}\cdot \mathcal{O}_{X'})^{c_{i}})=\mathcal{J}(f^\ast \phi)
$$ by the argument of [\ref{JM}, Proposition 1.9]. If $\{\varphi_{m}\}$ is a sequence of ideal functions which converges to $\varphi$ strongly in the norm, then $f^{\ast}\varphi_{m}$ is a decreasing sequence of ideal functions which converges to $f^{\ast}\varphi$ strongly in the norm by Proposition \ref{norm_reg}. Therefore we have $\mathcal{J}(\varphi)\cdot \mathcal{O}_{X'}=\mathcal{J}((1+\varepsilon)\varphi_{m})\cdot \mathcal{O}_{X'}=\mathcal{J}((1+\varepsilon)f^{\ast}\varphi_{m})=\mathcal{J}(f^{\ast}\varphi)$.
\end{proof}

Recall from [\ref{JM}] that if $\mathfrak{b}_\bullet$ is subadditive, then the limit
$$
v(\mathfrak{b}_\bullet):=\lim\limits_{m \rightarrow \infty}\frac{1}{m}v(\mathfrak{b}_m)\in [0,+\infty]
$$
is well-defined. For the purpose of constructing a "good" valuative function, we introduce the notion of a subadditive sequence of ideals of controlled growth as follows.

\begin{defn}[\ref{JM}, Definition 2.9]
A subadditive sequence of ideals $\mathfrak{b}_\bullet$ is \emph{of controlled growth} if
$$
\frac{v(\mathfrak{b}_t)}{t} > v(\mathfrak{b}_\bullet)-\frac{A(v)}{t}
$$
for every nontrivial tempered valuation $v$ and every $t>0$.
\end{defn}

We see that $v(\mathfrak{b}_\bullet):=\lim\limits_{m \rightarrow \infty}\frac{1}{m}v(\mathfrak{b}_m)< +\infty$ for every nontrivial tempered valuation $v$. Furthermore, if we define $\log|\mathfrak{b}_\bullet|(v)=-v(\mathfrak{\mathfrak{b}_\bullet})$, then $\log|\mathfrak{b}_\bullet|$ is approximated by $\frac{1}{m}\log|\mathfrak{b}_m|$ strongly in the norm and hence qpsh. Given a qpsh function, if we define $\mathcal{J}(\varphi)_t:=\mathcal{J}(t\varphi)$, then $\mathcal{J}(\varphi)_\bullet$ is a subadditive sequence of controlled growth by Proposition \ref{subadditivity}, Definition \ref{def_mult} and Remark \ref{rem_mult}. This allows us to give a characterization of qpsh functions as follows. Readers could compare the following theorem with Theorem \ref{ch_alg}.

\begin{thm}\label{ch_qpsh}
Let $\varphi$ be a bounded homogeneous function. Then the followings are equivalent.\\
(1). $\varphi$ is qpsh. \\
(2). There is a subadditive sequence of ideals $\mathfrak{b}_{\bullet}$ of controlled growth such that $\varphi=\log |\mathfrak{b}_\bullet|$.\\
(3). The ideal $\mathcal{J}(t\varphi)$ is nonzero for every $t>0$ and $\varphi=\log |\mathcal{J}(\varphi)_\bullet|$.
\end{thm}

\begin{proof}
If we assume (1), then (3) follows from the previous argument together with Definition \ref{def_mult} and Remark \ref{rem_mult}. Note that (3) implies (2) if we simply put $\mathfrak{b}_\bullet=\mathcal{J}(\varphi)_\bullet$. Finally, (1) follows from (2) by the previous argument.
\end{proof}

\begin{rem}\label{qpsh_ideal}
From the above theorem we see that every qpsh function $\varphi$ can be approximated by a decreasing sequence of ideal functions $\varphi_k$ in the norm. Indeed, we can take $\varphi_k= \frac{1}{2^k}\log|\mathcal{J}(2^k\varphi)|$. However, if $\varphi$ is only the pointwise limit of a decreasing sequence of ideal functions on $\mathrm{V}_X$, then $\varphi$ is not necessarily qpsh (cf. Example \ref{exa:non-psh}).
\end{rem}

An immediate application of the above discussion is the following result on the support of a qpsh function.

\begin{cor}\label{cor:supp}
Let $\varphi$ be a qpsh function. Then its support $\mathrm{Supp}\varphi$ is a countable union of proper Zariski closed subsets of $X$.
\end{cor}

\begin{rem}
Readers can compare the constructions here with [\ref{BFJ1}]. If we work on $X=\mathrm{Spec}\widehat{R}$ where $R$ is the localization of $\mathbb{C}[x_1,\ldots,x_n]$ at the origin, then our definition of qpsh functions coincides the notion of \emph{formal psh functions}. A brief argument is as follows. Given a formal psh function $g$, we have a subadditive sequence of ideals $\{\mathcal{L}^2(tg)\}_{t>0}$ in $\widehat{R}$ by [\ref{BFJ1}, Theorem 3.10] which satisfies that $v(\mathcal{L}^2(tg))+A(v) +(1+\epsilon)g(v) \geq 0$ for every quasi-monomial valuation $v$ centered at the origin and an arbitrary small $\epsilon$ by [\ref{BFJ1}, Theorem 3.9]. It follows that $\{\mathcal{L}^2(tg)\}_{t>0}$ form a subadditive sequence of ideals of controlled growth which induces to a qpsh function $\varphi$ on $X$. Therefore $\varphi(v)=g(v)$ for every divisorial valuation $v$ centered at the origin. Conversely, a qpsh function can be naturally viewed as a formal psh function by definition. Therefore we construct an one-to-one correspondence. The details are left to the readers.
\end{rem}

\begin{rem}
Recall from complex geometry that a function $\varphi:X \rightarrow [-\infty,+\infty)$ from a complex manifold is \emph{qpsh} if it is locally equal to the sum of a smooth function and a psh function. If $X$ is a smooth complex variety, then we should be able to define the valuative transform of $\varphi$ which is expected to be a qpsh function on the valuation space $\mathrm{V}_X$ as defined in this paper. This was done locally in [\ref{BFJ1}] and its predecessors [\ref{FJ1}], [\ref{FJ2}], [\ref{FJ3}]. However, the global situation is not fully understood by us at this point.
\end{rem}

\vspace{0.3cm}
%%%%%%%%%%%%%%%%%%%%%%%%%%%%%
\section{Computing norms}\label{c-norms}

\subsection{Generalities}

\begin{defn}
Let $\varphi$ be a bounded homogeneous function and $\mathfrak{q}$ be a nonzero ideal on $X$. We say a nontrivial tempered valuation $v \in \mathrm{V}_X^\ast$ \emph{computes} $\|\varphi\|_{\mathfrak{q}}$ if the equality $\|\varphi\|_{\mathfrak{q}}=\frac{|\varphi(v)|}{A(v)+v(\mathfrak{q})}$ holds.
\end{defn}

The main result of this section is the following theorem.

\begin{thm}\label{c_norms}
Let $\varphi\in \mathrm{QPSH}(X)$ be a qpsh function and let $\mathfrak{q}$ be a nonzero ideal on $X$. Then there exists a nontrivial tempered valuation $v$ which computes $\|\varphi\|_{\mathfrak{q}}$.
\end{thm}

Before we prove this theorem, we need some preparations.

\begin{prop}
Let $\varphi$ be a bounded homogeneous function that is determined on some dual complex $\Delta(Y,D)$ in the sense of $\varphi=\varphi \circ r_{Y,D}$. Assume that $\varphi$ is weakly continuous (cf. Definition \ref{defn:val_fun}). Then there exists a quasi-monomial valuation $v$ which computes $\|\varphi\|_{\mathfrak{q}}$. If we assume further that $\varphi$ is affine on each face of $\Delta(Y,D)$, then there exists a divisorial valuation $v$ which computes $\|\varphi\|_{\mathfrak{q}}$.
\end{prop}

\begin{proof}
For every nontrivial tempered valuation $v\in V_{X}^\ast$, we have
\begin{align*}
\frac{|\varphi(v)|}{A(v)+v(\mathfrak{q})}\geq \frac{|\varphi\circ r_{Y,D}(v)|}{A(r_{Y,D}(v))+r_{Y,D}(v)(\mathfrak{q})}
 \end{align*}
 with equality if and only if $v\in \mathrm{QM}(Y,D)$. Thus
$$
\|\varphi\|_{\mathfrak{q}}=\sup\limits_{v\in \mathrm{QM}(Y,D)} \frac{|\varphi(v)|}{A(v)+v(\mathfrak{q})}=\sup\limits_{v\in \Delta(Y,D)} \frac{|\varphi(v)|}{1+v(\mathfrak{q})}.
 $$
Since $\varphi$ is weakly continuous, the function $v\rightarrow \frac{|\varphi(v)|}{A(v)+v(\mathfrak{q})}$ is continuous on $\mathrm{QM}(Y,D)$. Therefore the function $v\rightarrow \frac{|\varphi(v)|}{1+v(\mathfrak{q})}$ is continuous on the dual complex $\Delta(Y,D)$ and hence achieves its maximum in $\Delta(Y,D)$.

Assume that $\varphi$ is affine on $\Delta(Y,D)$, and we denote by $D_{i}$'s the irreducible components of $D$. After replacing $(Y,D)$ by some higher log resolution, we can assume that $(Y,D)$ is a log resolution of $\mathfrak{q}$ by Lemma \ref{lem:linear}. Then we have $\|\varphi\|_{\mathfrak{q}}=\max\limits_{D_i} \frac{|\varphi(\mathrm{ord}_{D_{i}})|}{A(\mathrm{ord}_{D_{i}})+\mathrm{ord}_{D_{i}}(\mathfrak{q})}$ where $D_i$ runs over all irreducible components of $D$ since the functions $\varphi$, $A$ and $\log|\mathfrak{q}|$ are all affine on $\Delta(Y,D)$.
\end{proof}

\subsection{Computing norms of qpsh functions}

This subsection is devoted to the proof of Theorem \ref{c_norms}. The proof here follows from the strategy of [\ref{JM}]. We first consider the local case.

\begin{lem}\label{cn_1}
Let $(R,\mathfrak{m})$ be a local ring, let $\varphi\in \mathrm{QPSH}(\mathrm{Spec}R)$ be a qpsh function, and let $\mathfrak{q}$ be a nonzero ideal on $\mathrm{Spec}R$. We set $\lambda^{-1}=\|\varphi\|_{\mathfrak{q}}$ and assume that $\sqrt{(\mathcal{J}(\lambda\varphi):\mathfrak{q})}=\mathfrak{m}$. If we define another qpsh function $\psi=\max\{\varphi,p\log|\mathfrak{m}|\}$ for a sufficiently large integer $p$, then $\|\varphi\|_{\mathfrak{q}}=\|\psi\|_{\mathfrak{q}}$. Moreover, if a nontrivial tempered valuation $v$ computes $\|\psi\|_{\mathfrak{q}}$, then $v$ also computes $\|\varphi\|_{\mathfrak{q}}$.
\end{lem}

\begin{proof}
Since $\sqrt{((\mathcal{J}(\lambda\varphi):\mathfrak{q})}=\mathfrak{m}$, we have $\mathfrak{m}^{n}\cdot \mathfrak{q} \subseteq \mathcal{J}(\lambda\varphi)$ for some integer $n$. Set $\lambda'^{-1}=\|\varphi\|_{\mathfrak{m}^{n}\cdot\mathfrak{q}}$, it follows that $\lambda'>\lambda$ by Lemma \ref{norm_1}. Pick an integer $p>n/(\lambda'-\lambda)$, and fix a sufficiently small number $\varepsilon< \ll 1$ such that $p>n/((1-\varepsilon)\lambda'-\lambda)$. Observe that
$$
\|\psi\|_{\mathfrak{q}}=\sup_{v\in\mathrm{V}_{R}^{\ast}} \frac{\min\{-\varphi(v),pv(\mathfrak{m})\}}{A(v)+v(\mathfrak{q})}\geq \sup_{v\in\mathrm{V}_{\varepsilon}^\ast} \frac{\min\{-\varphi(v),pv(\mathfrak{m})\}}{A(v)+v(\mathfrak{q})}
$$
where $\mathrm{V}_{\varepsilon}^\ast$ is the set of $v\in \mathrm{V}_{R}^\ast$ satisfying $\frac{-\varphi(v)}{A(v)+v(\mathfrak{q})}\geq (1-\varepsilon)/ \lambda$.

By the definition of $\lambda'$ we have $\frac{nv(\mathfrak{m})}{-\varphi(v)}\geq \lambda'-\frac{A(v)+v(\mathfrak{q})}{-\varphi(v)}$ for every nontrivial tempered valuation $v$. This implies that
\begin{align*}
\|\psi\|_{\mathfrak{q}}&\geq \sup_{v\in\mathrm{V}_\varepsilon}\frac{-\varphi(v)}{A(v)+v(\mathfrak{q})}\min\{1,\frac{p}{n}(\lambda'-\frac{A(v)+v(\mathfrak{q})}{-\varphi(v)})\} \\
                       &\geq \sup_{v\in\mathrm{V}_\varepsilon} \frac{-\varphi(v)}{A(v)+v(\mathfrak{q})}\min\{1,\frac{p}{n}(\lambda'-\frac{\lambda}{1-\varepsilon})\}=\sup_{v\in\mathrm{V}_\varepsilon} \frac{-\varphi(v)}{A(v)+v(\mathfrak{q})}=\|\varphi\|_{\mathfrak{q}}.
\end{align*}
Moreover, if a nontrivial tempered valuation $v$ computes $\|\psi\|_{\mathfrak{q}}$, then from the above inequality we see that $v$ also computes $\|\varphi\|_{\mathfrak{q}}$.
\end{proof}

\begin{lem}\label{cn_2}
Let $(R,\mathfrak{m})$ be a local ring, let $\varphi$ be an ideal function on $\mathrm{Spec}R$ such that $\varphi\geq p\log|\mathfrak{m}|$ for some integer $p$, and let $\mathfrak{q}$ be a nonzero ideal on $\mathrm{Spec}R$. Then there exists a nontrivial tempered valuation $v\in \mathbb{V}_{R,\mathfrak{m},M}$ (cf. Definition \ref{defn:cpt-valsp}) which computes $\|\psi\|_{\mathfrak{q}}$ provided that $M>p\cdot \|\varphi\|_{\mathfrak{q}}^{-1}$.
\end{lem}

\begin{proof}
If we write $c=p/M$, then $0<c<\|\psi\|_{\mathfrak{q}}$. For every $v\in \mathbb{V}_{R,\mathfrak{m}}$ such that $\frac{-\varphi(v)}{A(v)+v(\mathfrak{q})}>c$ , we have $A(v)\leq A(v)+v(\mathfrak{q})<p/c=M$. Thus $\|\varphi\|_{\mathfrak{q}}=\sup\limits_{v\in \mathbb{V}_{R,\mathfrak{m},M}} \frac{-\varphi(v)}{A(v)+v(\mathfrak{q})}$. Note that $\mathbb{V}_{R,\mathfrak{m},M}$ is compact. Since the function $v\rightarrow \frac{-\varphi(v)}{A(v)+v(\mathfrak{q})}$ is usc as the valuative function $A$ is lsc, the maximum can be achieved in $\mathbb{V}_{R,\mathfrak{m},M}$.
\end{proof}

\begin{lem}\label{cn_3}
Let $\varphi \in \mathrm{QPSH}(X)$ be a qpsh function on $X$ and $\{\varphi_{m}\}$ be a decreasing sequence of algebraic functions which converges to $\varphi$ strongly in the norm. Set $\lambda^{-1}=\|\varphi\|_{\mathfrak{q}}$ and $\lambda_{m}^{-1}=\|\varphi_{m}\|_{\mathfrak{q}}$. Then,
$\mathcal{J}(\lambda \varphi) \subseteq \mathcal{J}(\lambda_{m} \varphi_{m})$ for every sufficiently large integer $m$.
\end{lem}

\begin{proof}
If $f\in \mathcal{J}(\lambda \varphi)$, then $\|\varphi\|_{f}<(1-\varepsilon)/\lambda$ for a sufficiently small number $\varepsilon>0$. Since $\lambda_{m}< \lambda /(1-\varepsilon)$ for sufficiently large $m$, we have $\|\varphi_{m}\|_{f}\leq \|\varphi\|_{f}<(1-\varepsilon)/\lambda <\lambda_{m}^{-1}$. It follows that $f\in \mathcal{J}(\lambda_{m} \varphi_{m})$ by Lemma \ref{norm_1}.
\end{proof}

\begin{lem}\label{cn_4}
Let $(R,\mathfrak{m})$ be a local ring, let $\varphi$ be a qpsh function on $\mathrm{Spec}R$ such that $\varphi\geq p\log|\mathfrak{m}|$, and let $\mathfrak{q}$ be a nonzero ideal on $\mathrm{Spec}R$. Then there exists a nontrivial tempered valuation $v\in \mathbb{V}_{R,\mathfrak{m},M}$ which computes $\|\varphi\|_{\mathfrak{q}}$ provided that $M>p\cdot \|\varphi\|_{\mathfrak{q}}^{-1}$.
\end{lem}

\begin{proof}
Assume that $\{\varphi_{m}\}$ is a decreasing sequence of ideal functions which converges to $\psi$ strongly in the norm. Then $\mathfrak{m}^{n}\cdot \mathfrak{q} \subseteq \mathcal{J}(\lambda \varphi) \subseteq \mathcal{J}(\lambda_{m} \varphi_{m})$ for every sufficiently large integer $m$ by Lemma \ref{cn_3}. We set $\lambda'^{-1}=\|\varphi\|_{\mathfrak{m}^{n}\cdot \mathfrak{q}}$ and $\lambda_{m}'^{-1}=\|\varphi_{m}\|_{\mathfrak{m}^{n}\cdot \mathfrak{q}}$. Note that $M>p \cdot \lambda_{m}$ for every sufficiently large integer $m$. Therefore for every sufficiently large integer $m$, there exists $v_{m} \in \mathbb{V}_{R,\mathfrak{m},M}$ which computes $\|\varphi_{m}\|_{\mathfrak{q}}$ by Lemma \ref{cn_2}. By passing $\{\varphi_{m}\}$ to a subsequence, we can assume $\{v_{m}\}$ is a sequence of points which converges to some point $v\in \mathbb{V}_{R,\mathfrak{m},M}$. Note that
\begin{align*}
\frac{-\lambda \varphi(v)}{A(v)+v(\mathfrak{q})}\geq \frac{-\lambda \varphi_{m}(v)}{A(v)+v(\mathfrak{q})} &\geq \frac{-\lambda \varphi_{m}(v_{n})}{A(v_{n})+v_{n}(\mathfrak{q})}-\delta  \\
&\geq 1-\|\lambda \varphi_{m}-\lambda_{n} \varphi_{n}\|_{\mathfrak{q}}-\delta \\
 &\geq 1-\lambda \| \varphi_{m}- \varphi_{n}\|_{\mathfrak{q}}-\delta -(\lambda_{n}-\lambda)\|\varphi\|_{\mathfrak{q}}
\end{align*}
where the second inequality holds because the function $v \rightarrow \frac{-\lambda \varphi_m(v)}{A(v)+v(\mathfrak{q})}$ is usc. Since $\| \psi_{m}-\psi_{n}\|_{\mathfrak{q}}$, $\delta$ and $\lambda_{n}-\lambda$ can be chosen arbitrary small, we have that $\frac{-\lambda \psi(v)}{A(v)+v(\mathfrak{q})}\geq 1$ and the conclusion follows.
\end{proof}

Now we turn to treat the global case.

\begin{proof}[Proof of Theorem \ref{c_norms}]
Pick a generic point $\xi$ of $V(\mathcal{J}(\lambda\varphi):\mathfrak{q})$. Note that $\|\varphi\|_{\mathfrak{q}}=\|\varphi_{\xi}\|_{\mathfrak{q}\cdot \mathcal{O}_{X,\xi}}$ by Lemma \ref{norm_2}(3). After replacing $X$ and $\varphi$ by $\mathrm{Spec}\mathcal{O}_{X,\xi}$ and $\varphi_{\xi}$, respectively, we reduce the global case to the local case. After replacing $\varphi$ by $\max\{\varphi,p\log|\mathfrak{m}|\}$ for a sufficiently large integer $p$ by Lemma \ref{cn_1}, we can assume that $\varphi\geq p\log|\mathfrak{m}|$. Finally by Lemma \ref{cn_4}, there exists a valuation $v\in \mathbb{V}_{X,\xi,M}$ which computes $\|\varphi\|_{\mathfrak{q}}$.
\end{proof}

An immediate consequence of Theorem \ref{c_norms} is the following corollary.

\begin{cor}\label{cor:ex_mult}
Let $\varphi$ be a qpsh function on $X$. Then, on every open subset $U$, we can explicitly write $\Gamma(U,\mathcal{J}(\varphi))=\{f\in \Gamma(U,\mathcal{O}_{X})|v(f)+A(v)+\varphi(v)>0$ for every $v\in \mathrm{V}_U^\ast\}$. Let $\mathfrak{q}$ be a nonzero ideal on $X$. Then, $\mathfrak{q} \subseteq \mathcal{J}(\varphi)$ if and only if $v(\mathfrak{q})+A(v)+\varphi(v)>0$ for every $v\in \mathrm{V}_{X}^{\ast}$.
\end{cor}

The following conjecture was raised as [\ref{JM}, Conjecture B] (cf. [\ref{JM}, Theorem 7.8]). It is already known for several special cases (cf. [\ref{JM}, Section 8 and 9]).

\begin{conj}\label{conj}
Let $\varphi$ be a qpsh function on $X$ and $\mathfrak{q}$ be a nonzero ideal on $X$. Then there exists a nontrivial quasi-monomial valuation $v$ which computes $\|\varphi\|_{\mathfrak{q}}$. Conversely, if a nontrivial tempered valuation $v$ computes the norm of some qpsh function, then $v$ is quasi-monomial.
\end{conj}

\vspace{0.3cm}
%%%%%%%%%%%%%%%%%%%%%%%%%%%%%
\section{Applications}

If $X$ is a smooth complex projective variety, then we are interested in associating a qpsh function to a line bundle which plays the role of a semi-positive singular metric. The starting point is the following easy observation. Given a pseudo-effective line bundle $L$, an ideal $\mathfrak{a}$ together with a nonnegative rational number $\lambda$ such that $L \otimes \mathfrak{a}^\lambda$ is semi-ample corresponds to a semi-positive singular metric $h$ in the sense that they give the same multiplier ideals $\mathcal{J}(\mathfrak{a}^{\lambda m})=\mathcal{J}(h^{\otimes m})$ for every integer $m>0$. However in general, this correspondence become quite mysterious since many analogue notions cannot be constructed. This has been studied in many relevant references such as [\ref{Bou}], [\ref{ELMNP}], [\ref{ELMNP-II}], [\ref{EP}], [\ref{Lehmann}], [\ref{Nakayama}]. We will discuss the qpsh function associated to a line bundle in detail within this section. Besides, it might be possible to generalize the results in this section to varieties with mild singularities such as klt singularities (cf. [\ref{BdFF}], [\ref{BdFU}]).

Throughout this section $X$ will be a projective smooth variety over $\mathbb{C}$ for simplicity. The term 'divisor' will always refer to a $\mathbb{Q}$-Cartier $\mathbb{Q}$-divisor. Given a section $s\in H^0(X,L)$ of a line bundle, the notation $\log|s|$ denotes the qpsh function defined locally by a regular function corresponding to $s$.

\subsection{$D$-psh functions}\label{sec:D-psh}

\begin{defn}\label{def:D-psh}
Let $D$ be a divisor. We define the set
$$
\mathcal{L}_{D}:=\{ \frac{1}{k} \log|\mathfrak{a}| ~| kmD \otimes \mathfrak{a}^{m}~\emph{is globally generated for every sufficiently divisible}~m\}.
$$
We then define set of $D$-psh functions to be the closure $\mathrm{PSH}(D)=\overline{\mathcal{L}_{D}}$ in the norm.
\end{defn}

\begin{lem}\label{lem:D-psh}
(1). $\mathrm{PSH}(D)$ is compact and convex in $\mathrm{QPSH}(X)$;\\
(2). $\mathrm{PSH}(tD)=t\mathrm{PSH}(D)$ for any $t\in \mathbb{Q}_{>0}$; \\
(3). $\mathrm{PSH}(D)+\mathrm{PSH}(D')\subseteq \mathrm{PSH}(D+D')$. \\
(4). If $A$ is a semiample divisor, then $\mathrm{PSH}(D) \subseteq \mathrm{PSH}(D+A)$.
\end{lem}

\begin{proof}
We firstly prove (1). To prove that $\mathrm{PSH}(D)$ is convex, it suffices to show that $\mathcal{L}_{D}$ is convex. Given qpsh functions $\varphi$, $\varphi'\in \mathcal{L}_{D}$ and a rational number $0<\lambda<1$, we will show that $\lambda \varphi + (1-\lambda) \varphi' \in \mathcal{L}_{D}$. If we write $\varphi = \frac{1}{k} \log |\mathfrak{a}|$, $\varphi' = \frac{1}{k'} \log |\mathfrak{a}'|$ and $\lambda =q/p$, then
\begin{align*}
\lambda \varphi + (1-\lambda) \varphi' &= \frac{1}{kp} \log|\mathfrak{a}^{q}| +\frac{1}{k'p} \log|\mathfrak{a}'^{p-q}| \\
&= \frac{1}{kk'p} \log|\mathfrak{a}^{qk'}\cdot \mathfrak{a}'^{k(p-q)}|.
\end{align*}
It is easy to check that $kk'pmL \otimes \mathfrak{a}^{mqk'}\cdot \mathfrak{a}'^{mk(p-q)}$ is globally generated for every sufficiently divisible integer $m$ and the conclusion follows. Note that (2), (3) and (4) can be proved in a similar way.
\end{proof}

\begin{quest}
Let $\varphi$ be a general qpsh function. Does there exist a divisor $D$ such that $\varphi \in \mathrm{PSH}(D)$?
\end{quest}

\begin{defn}\label{def:pD-psh}
The set of pseudo $D$-psh functions is defined to be $\mathrm{PSH}_{\sigma}(D):= \bigcap\limits_{\varepsilon >0} \mathrm{PSH}(D+ \varepsilon A)$ where $A$ is an ample divisor.
\end{defn}

Note that the above definition is independent of the choice of the ample divisor $A$, and that the set $\mathrm{PSH}_{\sigma}(D)$ also satisfies the properties listed in Lemma \ref{lem:D-psh}.

\begin{thm}[Nadel Vanishing]\label{thm:nad}
Let $L$ be a line bundle on a smooth projective variety $X$ and $L\equiv A+D$ where $A$ is a nef and big $\mathbb{Q}$-divisor. Assume that $\varphi \in \mathrm{PSH}_{\sigma}(D)$. Then $$ H^{i}(X,(K_{X}+L)\otimes \mathcal{J}(\varphi))=0 $$ for all $i>0$.
\end{thm}

\begin{proof}
First by Kodaira Lemma $A-\delta E$ is ample for some effective divisor $E$ and every sufficiently small number $\delta>0$. If we write $\varphi_E=\log|\mathcal{O}_X(-E)|$, then by semicontinuity of multiplier ideals we have $\mathcal{J}(\varphi)=\mathcal{J}(\varphi+\delta \varphi_E)$ for every sufficiently small number $\delta>0$. After replacing $A$ and $\varphi$ with $A-\delta E$ and $\varphi+\delta \varphi_E$, respectively, we can assume that $A$ is ample.

By definition we can assume that there exists a sequence of ideal functions $\{\varphi_{k}\}$ which converges to $\varphi$ strongly in the norm, such that $\varphi_{k} \in \mathcal{L}_{D+ \epsilon_{k}A}$ and $ \epsilon_{k} \rightarrow 0+$. Choose $\varepsilon \ll 1$ such that $A - \varepsilon D$ is ample. We see that $\mathcal{J}(\varphi)=\mathcal{J}((1+\varepsilon) \varphi_{k})$ for every sufficiently large integer $k$ by Lemma \ref{qpsh_mult}. Note that $(1+\varepsilon)\varphi_{k} \in \mathcal{L}_{(1+\varepsilon)(D+\epsilon_{k}A)}$. For a sufficient large integer $k$, $A-\varepsilon D - (1+\varepsilon) \epsilon_{k}A$ is ample. After replacing $A$ and $\varphi$ by $A-\varepsilon D - (1+\varepsilon) \epsilon_{k}A$ and $(1+\varepsilon)\varphi_{k}$, respectively, we reduce to the classical Nadel vanishing (cf. [\ref{Laz}]).
\end{proof}

As an application of the above theorem, one can easily deduce the following theorem by letting $G=K_{X}+(n+1)H$ where $H$ is a hypersurface of $X$ and $n=\dim X$, with the aid of the Castelnuovo-Mumford regularity.

\begin{thm}[Global generation]\label{thm:gg}
Let $D$ be a divisor on $X$ and $\varphi$ be a qpsh function. Then, $\varphi \in \mathrm{PSH}_{\sigma}(D)$ if and only if there exists a line bundle $G$ such that $(mD+G)\otimes \mathcal{J}(m\varphi)$ is globally generated for all $m\in \mathbb{Z}_{+}$ with $mD$ integral.
\end{thm}

Given a qpsh function $\varphi$, a positive real number $\lambda$ is said to be the (higher) jumping number of $\varphi$ if $\mathcal{J}((\lambda-\epsilon)\varphi)\supsetneq \mathcal{J}(\lambda\varphi)$ for every positive real number $\epsilon$.

\begin{defn}
Let $\varphi$ be a qpsh function. We define the ideal $\mathcal{J}_-(\varphi)$ to be the largest ideal in the set $\{\mathfrak{a}|\|\log|\mathfrak{a}|-\varphi\|\leq 1\}$. One can see that $\mathcal{J}_-(\varphi)$ can be written explicitly as
$$
\Gamma(U,\mathcal{J}_-(\varphi))=\{f\in\mathcal{O}_X(U)|v(f)+A(v)+\varphi(v)\geq 0~\emph{for every }v\in V_U^\ast\}
$$
for every open subset $U$.
\end{defn}

\begin{lem}
If $\varphi$ is $D$-psh for some divisor $D$, then the descending chain of ideals $\mathcal{J}((1-\epsilon)\varphi)$ stabilizes as $\epsilon \rightarrow 0+$. Further, $\mathcal{J}((1-\epsilon)\varphi)=\mathcal{J}_-(\varphi)$ for $\epsilon \ll 1$. It follows that the set of its (higher) jumping numbers is discrete.
\end{lem}

\begin{proof}
By adding an ample divisor to $D$, we can assume that $D$ is Cartier. By Theorem \ref{thm:nad} and the Castelnuovo-Mumford regularity there exists an ample line bundle $G$ such that $\mathcal{O}_X(D+G)\otimes \mathcal{J}((1-\epsilon)\varphi)$ is globally generated for $\epsilon \ll 1$. Since the descending chain of vector spaces $H^0(X,\mathcal{O}_X(D+G)\otimes \mathcal{J}((1-\epsilon)\varphi))$ will stabilize as $\epsilon \rightarrow 0+$, the descending chain of ideals $\mathcal{J}((1-\epsilon)\varphi)$ will stabilize. The reader can find more details in [\ref{Lehmann}, Theorem 4.2].

Fix a sufficiently small number $\epsilon'$. Since $\|\log|\mathcal{J}((1-\epsilon')\varphi)|-(1-\epsilon)\varphi\|<1$ for every sufficiently small number $\epsilon$, we see that $\|\log|\mathcal{J}((1-\epsilon')\varphi)|-\varphi\|\leq 1$. It follows that $\mathcal{J}((1-\epsilon)\varphi)\subseteq \mathcal{J}_-(\varphi)$. To prove the converse inclusion, simply notice that
$$
\Gamma(U,\mathcal{J}_-(\varphi))=\{f\in\mathcal{O}_X(U)|v(f)+A(v)+\varphi(v)\geq 0~\emph{for every }v\in V_U^\ast\}
$$
and hence $\mathcal{J}((1-\epsilon)\varphi)\supseteq \mathcal{J}_-(\varphi)$ for $\epsilon \ll 1$ by Corollary \ref{cor:ex_mult}.
\end{proof}

To investigate the structure of the sets $\mathrm{PSH}(D)$ and $\mathrm{PSH}_\sigma(D)$, we need the following construction. Given an integer $k$, a divisor $D$ and a qpsh function $\varphi$, we define the linear system $V_{m}(D,\varphi,t):=
\{L\in |\llcorner mD \lrcorner||\frac{1}{m}\log|s_L|\leq \frac{1}{t}\log|\mathcal{J}_-(t\varphi)|\}$ where $s_L$ is the section associated to the divisor $L$ and $\epsilon \ll 1$. If we set $\mathfrak{a}(D,\varphi,t)_m:=\mathfrak{b}(V_m(D,\varphi,t))$ where $\mathfrak{b}(V_m(D,\varphi,t))$ denotes the base ideal of the linear system $V_m(D,\varphi,t)$, then $\mathfrak{a}(D,\varphi,t)_\bullet$ is a graded sequence of ideals. Moreover, we define $\varphi_{t}^{D}:=\log|\mathfrak{a}(D,\varphi,t)_\bullet|$ for every positive rational number $t$.

\begin{lem}\label{lem:D_psh}
Let $D$ be a divisor on $X$ and $\varphi$ be a qpsh function. Then, $\varphi \in \mathrm{PSH}(D)$ if and only if $\varphi= \lim\limits_{t \rightarrow \infty} \varphi_{t}^{D}$ pointwisely.
\end{lem}

\begin{proof}
First assume that $\varphi \in \mathrm{PSH}(D)$. Let $\{\varphi_{m}\}$ be a sequence of ideal functions which converges to $\varphi$ such that each $\varphi_{m} \in \mathcal{L}_{D}$. If $t$ is not a (higher) jumping number of $\varphi$, then, by Lemma \ref{qpsh_mult} we have
$$
\mathcal{J}_-(t\varphi)=\mathcal{J}((t-\epsilon)\varphi)=\mathcal{J}((t-\epsilon+\epsilon')\varphi_m) \supseteq \mathcal{J}_-(t\varphi_{m})
$$
and
$$
\mathcal{J}_-(t\varphi)=\mathcal{J}(t\varphi)=\mathcal{J}((t+\epsilon)\varphi_m) \subseteq \mathcal{J}_-(t\varphi_m)
$$
 for every sufficiently large integer $m$. It follows that $\mathcal{J}_-(t\varphi)=\mathcal{J}_-(t\varphi_m)$ and $\varphi_{t}^{D}=\varphi_{m,t}^{D}$. Note that $\varphi_{m.t}^{D} \geq \varphi_{m}$, and hence $\frac{1}{t}\log|\mathcal{J}_-(t\varphi)| \geq \varphi_{t}^{D} \geq \varphi$. If $t$ is a (higher) jumping number, then $\varphi_t^D \geq \varphi_{t-\epsilon}^D \geq \varphi$. Therefore, we have $\|\varphi_t^D -\varphi \| \leq \frac{1}{t}$ and hence $\varphi = \lim\limits_{t \rightarrow \infty} \varphi_{t}^{D}$. \\
Conversely, we assume that $\varphi= \lim\limits_{t \rightarrow \infty} \varphi_{t}^{D}$. Since $\varphi_{t}^{D}$ is algebraic from $\mathfrak{a}(D,\varphi,t)_{\bullet}$ for each $t$, $\varphi_t^D$ is $D$-psh for every $t>0$. Since $\frac{1}{t}\log|\mathcal{J}_-(t\varphi)| \geq \varphi_t^D$ and $\varphi_t^D$ has a decreasing subsequence, $\varphi_t^D$ converges to $\varphi$ strongly in the norm which implies the conclusion immediately.
\end{proof}

For every nontrivial tempered valuation $v$, we define $v(\|D\|)=v(\mathfrak{a}_{\bullet})$ where $\mathfrak{a}_{m}=\mathfrak{b}(|\llcorner mD \lrcorner|)$.

\begin{prop}\label{prop:max}
The set $\mathrm{PSH}(D)$ is closed under taking the supremum. The maximal $D$-psh function $\varphi_{\max}$ can be written explicitly as $\varphi_{\max}(v)=-v(\|D\|)$ for all $v\in \mathrm{V}_{X}^\ast$.
\end{prop}

\begin{proof}
Let $\varphi_{\lambda}$ be a family of $D$-psh functions. By Lemma \ref{lem:D_psh} $\varphi_{\lambda}= \lim\limits_{t \rightarrow \infty} \varphi_{\lambda,t}^{D}$. Note that $\varphi_{\lambda,t}^{D}=\log|\mathfrak{a}(D,\varphi_\lambda,t)_{\bullet}|$, where $\mathfrak{a}(D,\varphi_\lambda,t)_{m}=\mathfrak{b}(V_{m}(D,\varphi_{\lambda},t))$. \\
If we write $\varphi=\sup_{\lambda} \varphi_{\lambda}$, then $\mathcal{J}_-(t\varphi_{\lambda}) \subseteq \mathcal{J}_-(t\varphi)$ for every $\lambda$ and every $t$. It follows that $\mathfrak{b}(V_m(D,\varphi_\lambda,t)) \subseteq \mathfrak{b}(V_m(D,\varphi,t))$ for every $m$, $\lambda$ and $t$. We deduce that $\sup_{\lambda} \varphi_{\lambda,t}^{D} \leq \varphi_{t}^{D}$ and hence
\begin{align*}
\varphi(v)&= \sup\limits_{\lambda} \lim\limits_{t \rightarrow \infty} \varphi_{\lambda,t}^{D}(v)  \\
&\leq \lim\limits_{t \rightarrow \infty} \sup\limits_{\lambda}\varphi_{\lambda,t}^{D}(v) \\
&\leq \lim\limits_{t \rightarrow \infty} \varphi_{t}^{D}(v)
\end{align*}
for every $v \in \mathrm{V}_X^\ast$. Note that the pointwise limits appeared in the above inequality exist because we can take decreasing subsequences which are bounded from below. Since $\frac{1}{t}\log|\mathcal{J}_-(t\varphi)| \geq \varphi_{t}^{D}$, we obtain the equality $\varphi = \lim\limits_{t \rightarrow \infty} \varphi_{t}^{D}$ and $\varphi$ is $D$-psh by Lemma \ref{lem:D_psh}.\\
Now we prove that $\varphi_{\max}(v)=-v(\|D\|)$ for all $v\in V_{X}^\ast$. Let $\varphi$ be a qpsh function such that $\varphi(v)=-v(\|D\|)$. Because $\varphi$ is algebraic from $\mathfrak{a}_{\bullet}$ where $\mathfrak{a}_{m}=\mathfrak{b}(|\llcorner mD \lrcorner|)$, $\varphi$ is $D$-psh. It suffices to show that $\varphi_{\max} \leq \varphi$. For each $t$, $\varphi_{\max,t}^{D}=\log|\mathfrak{a}(D,\varphi_{\max},t)_{\bullet}|$ where $\mathfrak{a}(D,\varphi_{\max},t)_{m}=\mathfrak{b}(V_{m}(D,\varphi_{\max},t))$. It follows that $\mathfrak{a}(D,\varphi_{\max},t)_{m} \subseteq \mathfrak{a}_{m}$, and hence $\varphi_{\max,t}^{D} \leq \varphi$. Therefore, $\varphi_{\max}=\lim\limits_{t \rightarrow \infty} \varphi_{\max,t}^{D} \leq \varphi$ which forces $\varphi_{\max}=\varphi$.
\end{proof}

For every nontrivial tempered valuation $v$, we define $\sigma_{v}(\|D\|):= \lim\limits_{\varepsilon \rightarrow 0+} v(\|D+ \varepsilon A\|)$ for some ample divisor $A$. Note that [\ref{Nakayama}] verifies that this definition is independent of the choice of the ample divisor $A$.

\begin{prop}\label{prop:max'}
The set $\mathrm{PSH}_{\sigma}(D)$ is closed under taking the supremum. The maximal pseudo $D$-psh function $\phi_{\max}$ can be expressed explicitly as $\phi_{\max}(v)=-\sigma_{v}(\|D\|)$ for all $v\in V_{X}^\ast$.
\end{prop}

\begin{proof}
Let $\varphi_{\lambda}$ be a family of pseudo $D$-psh functions, and let $\varphi=\sup_\lambda \varphi_{\lambda}$. By Theorem \ref{thm:gg} there exists an ample divisor $G$ such that $\varphi_{\lambda,k} \in \mathrm{PSH}(D+\frac{1}{k}G)$ where $\varphi_{\lambda,k}= \frac{1}{k} \log |\mathcal{J}(k\varphi_{\lambda})|$. We have $\sup_{\lambda} \varphi_{\lambda,k} \in \mathrm{PSH}(D+\frac{1}{k}G)$ for every $k$ by Proposition \ref{prop:max}. Since $\Sigma_\lambda \mathcal{J}(k\varphi_\lambda) \subseteq \mathcal{J}(k\varphi)$, we have $\varphi_{k} \geq \sup_{\lambda} \varphi_{\lambda,k} \geq \varphi$. It follows that $$\varphi=\lim\limits_{k \rightarrow \infty}(\sup_{\lambda} \varphi_{\lambda,k}) \in \mathrm{PSH}_{\sigma}(D).$$
Now we prove that $\phi_{\max}(v)=-\sigma_{v}(\|D\|)$ for all $v\in V_{X}^\ast$. Let $\phi(v)=-\sigma_{v}(\|D\|)$, and let $\varphi_{\max}^\epsilon$ be the maximal $(D+ \epsilon A)$-psh function for every $\epsilon \ll 1$. We see that $\phi= \lim\limits_{\epsilon \rightarrow 0+}\varphi_{\max}^\epsilon$ pointwisely. Because $\varphi_{\max}^\epsilon$ is decreasing as $\epsilon \rightarrow 0+$, $\mathcal{J}(m\varphi_{\max}^\epsilon)$ form a descending chain of ideals as $\epsilon \rightarrow 0+$ for every integer $m>0$. If we fix an integer $m$ and a sequence $\epsilon_1 >\epsilon_2 >\ldots$ such that $\lim\limits_{k\rightarrow \infty} \epsilon_k=0$, then the descending chain stabilizes when $k \gg 0$ because there exists an ample divisor $G$ such that $mD+G$ is Cartier and $\mathcal{O}_X(mD+G)\otimes \mathcal{J}(m\varphi_{\max}^{\epsilon_k})$ is globally generated for every $k \gg 0$. It follows that $\|\varphi_{\max}^{\epsilon_k}-\varphi_{\max}^{\epsilon_{k'}}\|<\frac{1}{m}$ for all sufficiently large $k$ and $k'$. Equivalently, $\varphi_{\max}^{\epsilon_k}$ form a Cauchy sequence with respect to the norm. Therefore $\varphi_{\max}^{\epsilon_k}$ converges to $\phi$ strongly in the norm, and hence $\phi \in \mathrm{PSH}_\sigma(D)$. Note that $\phi_{\max} \leq \varphi_{\max}^\epsilon$, and hence $\phi_{\max} \leq \phi$ which implies the conclusion.
\end{proof}

The question below was raised by B. Lehmann (cf. [\ref{Lehmann}, Question 6.15]).

\begin{quest}
Is the maximal pseudo $D$-psh function algebraic?
\end{quest}

Abundant divisors, introduced by $[\ref{Nakayama}]$ and $[\ref{BDPP}]$, form a class of pseudo-effective divisors with nice asymptotic behavior. We denote by $\kappa_\sigma(D)$ the numerical Kodaira dimension. A pseudo-effective divisor $D$ is said to be \emph{abundant} if $\kappa(D)=\kappa_\sigma(D)$. We present the following easy corollary for the reader's convenience.

\begin{cor}
(1). The set $\mathrm{PSH}(D)$ is nonempty if and only if $D$ is $\mathbb{Q}$-effective.\\
(2). $0 \in \mathrm{PSH}(D)$ if and only if $D$ is nef and abundant;\\
(3). The set $\mathrm{PSH}_\sigma(D)$ is nonempty if and only if $D$ is pseudo-effective. \\
(4). $0 \in \mathrm{PSH}_{\sigma}(D)$ if and only if $D$ is nef.\\
(5). Let $\varphi_{\max}$ be the maximal $D$-psh function, and $\phi_{\max}$ be the maximal pseudo $D$-psh fucntion. Then, $D$ is abundant if and only if $\varphi_{\max}$=$\phi_{\max}$.
\end{cor}

\begin{proof}
The first statement is trivial. The second is a consequence of the main result of [\ref{Russo}], and (4) follows from (2) immediately. If $D$ is not pseudo-effective, then $\mathrm{PSH}_\sigma(D)$ is empty from (1). We prove (3) as follows. If $D$ pseudo-effective, then $\mathrm{PSH}_\sigma(D)$ is nonempty by Proposition \ref{prop:max'}. To prove (5), simply notice that $D$ is abundant if and only if $v(\|D\|)=\sigma_v(\|D\|)$ for every divisorial valuation $v$ by [\ref{Lehmann}, Proposition 6.18] and the last statement follows by Proposition \ref{prop:max} and Proposition \ref{prop:max'}.
\end{proof}

\begin{quest}
Assume that the divisor $D$ is abundant. Is the set $\mathrm{PSH}(D)$ equal to the set $\mathrm{PSH}_\sigma (D)$?
\end{quest}

We introduce the following definition of the perturbed ideal and the diminished ideal as [\ref{Lehmann}, Definition 4.3 and Definition 6.2]. We use the notation $\mathcal{J}_{\sigma,-}(D)$ instead of $\mathcal{J}_{-}(D)$ to avoid that readers may confuse it with the notation $\mathcal{J}_-(\varphi)$.

\begin{defn}
Let $D$ be a pseudo-effective divisor. We define the perturbed ideal $\mathcal{J}_{\sigma,-}(D)$ to be the smallest ideal in the finite descending chain $\{\mathcal{J}(\|L+\frac{1}{m} A\|)\}_{m=1}^\infty$, and we define the diminished ideal $\mathcal{J}_\sigma (D)$ to be the largest ideal in the set $\{\mathcal{J}_{\sigma,-}((1+\epsilon)D)\}_{\epsilon>0}$.
\end{defn}

Finally, we obtain a generalization of [\ref{Lehmann}, Theorem 6.14].

\begin{thm}\label{c_max_norm}
Let $D$ be a pseudo-effective divisor. Assume that $\phi_{\max}$ is the maximal pseudo $D$-psh function. Then, the perturbed ideal $\mathcal{J}_{\sigma,-}(D)=\mathcal{J}_{-}(\phi_{\max})$, and the diminished ideal $\mathcal{J}_\sigma(D)=\mathcal{J}(\phi_{\max})$. In particular, we can write $\mathcal{J}_\sigma(D)$ explicitly as $\Gamma(U,\mathcal{J}_{\sigma}(L))=\{f\in \Gamma(U,\mathcal{O}_{X})|v(f)+A(v)-\sigma_{v}(\|L\|)>0$ for all $v\in \mathrm{V}_U^\ast\}$. Further, a nonzero ideal $\mathfrak{q} \subseteq \mathcal{J}_{\sigma}(\|L\|)$ if and only if $v(\mathfrak{q})+A(v)-\sigma_{v}(\|L\|)>0$ for all $v\in \mathrm{V}_{X}^{\ast}$.
\end{thm}

\begin{proof}
The equality $\mathcal{J}_{\sigma,-}(D)=\mathcal{J}_{-}(\phi_{\max})$ follows from [\ref{Lehmann}, Proposition 4.7]. To prove the second equality, note that by definition $\mathcal{J}_\sigma(D)=\mathcal{J}((1+\epsilon)\varphi_{\max}^\delta)$, where $\varphi_{\max}^\delta$ denotes the maximal $(D+\delta A)$-psh function for an ample divisor $A$, for sufficiently small $\epsilon$ and sufficiently small $\delta=\delta(\epsilon)$. From the proof of Proposition \ref{prop:max'}, $\varphi_{\max}^\delta$ converges to $\phi_{\max}$ strongly in the norm. Therefore, Lemma \ref{qpsh_mult} asserts that $\mathcal{J}(\phi_{\max})=\mathcal{J}((1+\epsilon)\varphi_{\max}^\delta)=\mathcal{J}_\sigma(D)$ as $\delta \rightarrow 0+$. The last statement is obvious by Corollary \ref{cor:ex_mult}.
\end{proof}

\begin{rem}
It should not be too difficult to generalize most results in this subsection from $\mathbb{Q}$-divisors to $\mathbb{R}$-divisors, that is, one can define $D$-psh functions for an $\mathbb{R}$-divisor $D$ and obtain similar results.
\end{rem}

\subsection{Finite generation}

The goal of this subsection is to prove the finite generation proposition below as an application of qpsh functions. For definitions and properties of different types of Zariski decompositions, divisorial algebras and modules, we refer to [\ref{Nakayama}].

\begin{prop}\label{prop:f_g}
Let $(X,B)$ be a log canonical pair. Assume that $K_X+B$ is $\mathbb{Q}$-Cartier and abundant, and that $R(K_X+B)$ is finitely generated. Then, for any reflexive sheaf $\mathscr{F}$, $M_{\mathscr{F}}^p(K_X+B)$ is a finitely generated $R(K_X+B)$-module.
\end{prop}

Before we prove the above proposition, we first prove the following lemma.

\begin{lem}[Global division]\label{lem:gl_div}
Let $X$ be a smooth projective variety of dimension $n$. Consider line bundles $L$ and $D$, a linear system $V \subseteq |L|$ which is spanned by the sections $\{s_1,\ldots,s_l\}$, and a $D$-psh function $\varphi$. If we denote by $\phi_V$ the $L$-psh function $\max\limits_{1\leq j \leq l}{\log|s_j|}$. Then, for every integer $m\geq n+2$, any section $\sigma$ in
$$
H^0(X,\mathcal{O}_X(K_X+mL+D)\otimes \mathcal{J}(m\phi_V+\varphi))
$$
can be written as a linear combination $\sum_j s_j g_j $ of sections $g_j$ in $H^0(X,\mathcal{O}_X(K_X+(m-1)L+D))$.
\end{lem}

\begin{proof}
Let $\{\varphi_k \in \mathcal{L}_D\}$ be a sequence of ideal functions which converges to $\varphi$ strongly in the norm. Since $\mathcal{J}(m\phi_V+\varphi_k) \supseteq \mathcal{J(}m\phi_V+\varphi)$, the section $\sigma$ vanishes along the ideal $\mathcal{J}(m\phi_V+\varphi_k)$. If we denote by $\mathfrak{a}$ the base ideal $\mathfrak{b}(V)$, then $\phi_V=\log|\mathfrak{a}|$. Apply [\ref{EP}, Theorem 4.1] and we deduce the conclusion.
\end{proof}

\begin{rem}
In the statement of [\ref{EP}, Theorem 4.1], one can verify that the assumption that $D \otimes \mathfrak{b}^\lambda$ is nef and abundant implies that $\lambda \log|\mathfrak{b}|$ is $D$-psh. Note that Lemma \ref{lem:gl_div} is not a generalization of [\ref{EP}, Theorem 4.1] because we did not obtain that every $g_j$ is in $H^0(X,\mathcal{O}_X(K_X+(m-1)L+D)\otimes \mathcal{J}((m-1)\phi_V+\varphi))$. Nonetheless, it should be possible to generalize in the sense that $g_j \in H^0(X,\mathcal{O}_X(K_X+(m-1)L+D)\otimes \mathcal{J}((m-1)\phi_V+\varphi))$, if one can develop a theory on the restriction of qpsh functions to subvarieties (see the proof of [\ref{EP}, Theorem 3.2]).
\end{rem}

\begin{proof}[Proof of Proposition \ref{prop:f_g}]
We can assume that $(X,B)$ is log smooth of dimension $n$, $K_X+B$ is a $\mathbb{Q}$-Cartier $\mathbb{Q}$-divisor, and $\mathscr{F}=\mathcal{O}_X(A)$ is a very ample line bundle by [\ref{Birkar}, Theorem 1.1]. Since $R=R(K_X+B)$ is finitely generated, after a possible truncation we can assume that $R$ is generated by $R_1=H^0(m_0(K_X+B))$ for some integer $m_0$ such that $m_0(K_X+B)$ is Cartier (see [\ref{Birkar}, Remark 2.2 and 2.3]). If we set $\mathfrak{a}=\mathfrak{b}(|m_0(K_X+B)|)$ and $L:=m_0(K_X+B)$, then $\phi:=\log|\mathfrak{a}|$ is the maximal $L$-psh function. The rest of the proof is an analogue of [\ref{DHP}, Section 3]. Let $m$ be a sufficiently large integer (to be specified later), and let $\sigma$ be a global section of $m(K_X+B)+A$. We have that
$$
m(K_X+B)+A=K_X+(n+2)L+D
$$
where $D:=B+(m-(n+2)m_0-1)(K_X+B+\frac{1}{m}A)+\frac{m_0(n+2)+1}{m}A$. If we set
$$
\varphi=\psi_m+(m-(n+2)m_0-1)\varphi_m
$$
where $\psi_m$ is $(B+\frac{m_0(n+2)+1}{m}A)$-psh such that $\|\psi_m\| <1$, and $\varphi_m$ is the maximal $(K_X+B+\frac{1}{m}A)$-psh function. Notice that
$$
\|\log|\sigma| - (n+2)\phi - \varphi \|^+\leq \|(m_0(n+2)+1)\varphi_m -(n+2)\phi -\psi_m\|^+.
$$
We will show that $(m_0(n+2)+1)\varphi_m \leq (n+2)\phi$ for $m$ sufficiently large which implies that $\|\log|\sigma| - (n+2)\phi - \varphi \|^+<1$ and by definition $\sigma$ vanishes along $\mathcal{J}((n+2)\phi+\varphi)$. Since $\phi$ is determined on some dual complex $\Delta(Y,D)$, it suffices to prove that $(m_0(n+2)+1)\varphi_m \leq (n+2)\phi$ on $\Delta(Y,D)$. Further, we can assume that $\phi$ is affine on $\Delta(Y,D)$. It suffices to check the above inequality at vertices because $\varphi_m$ is convex on the dual complex. From the argument of Proposition \ref{prop:max'}, we see that $m_0\varphi_m$ converges to $\phi$ strongly in the norm since $K_X+B$ is abundant. Therefore for $m$ sufficiently large the inequality $\frac{m_0(n+2)+1}{n+2}\varphi_m \leq \phi$ holds at vertices of $\Delta(Y,D)$, and hence for every nontrivial tempered valuation. Finally, $\sigma$ can be written as a linear combination $\sum_j s_j g_j $ of sections $g_j$ in $H^0(X,\mathcal{O}_X((m-m_0)(K_X+B)+A)$ by Lemma \ref{lem:gl_div}, which completes the proof.
\end{proof}

\begin{rem}\label{rem:mmp}
The above finite generation proposition can be proved in another way as follows. Since the conclusion that $M_{\mathscr{F}}^p(K_X+B)$ is a finitely generated $R(K_X+B)$-module is equivalent to that $(X,B)$ has a good minimal model by [\ref{Birkar}, Theorem 1.3], it suffices to prove that $(X,B)$ has a good minimal model. By [\ref{BH-I}, Theorem 5.3] we conclude that $(X,B)$ has a log minimal model $(X',B')$. Since the positive part of the CKM-Zariski decomposition is semi-ample, the log minimal model $(X,B)$ is good. We here give a different proof without using the minimal model theory, in particular the length of extremal rays.
\end{rem}

Proposition \ref{prop:f_g} can be slightly generalized as follows.

\begin{defn}[\ref{BCHM}, Definition 3.6.4 and 3.6.6]
Let $D$ be a divisor on $X$. A normal projective variety $Z$ is said to be the \emph{ample model} of $D$ if there is a rational map $g:X \dashrightarrow Z$ and an ample $\mathbb{R}$-divisor $H$ on $Z$ such that if $p:W \rightarrow X$ and $q: W \rightarrow Z$ resolve $g$ then $q$ is a contraction and we can write $p^\ast D = q^\ast H+ N$, where $N \geq 0$ is an $\mathbb{R}$-divisor and for every $B \sim_\mathbb{Q} p^\ast D$ then $B \geq N$. Let $(X,B)$ be a pair. A normal variety $Z$ is said to be the \emph{log canonical model} of $(X,B)$ if it is the ample model of $K_X+B$.
\end{defn}

\begin{lem}\label{lem:f_g}
Let $D$ be an abundant divisor on a normal projective variety $X$. Assume that $D$ has the ample model. Then, $R(D)$ is finitely generated.
\end{lem}

\begin{proof}
After replacing $X$ by a log resolution, we can assume that $g:X \dashrightarrow Z$ is a morphism and $D=P+N= g^\ast H+ N$ where $H$ is an ample $\mathbb{R}$-divisor on the ample model $Z$ and $N \geq 0$ is an $\mathbb{R}$-divisor such that for every $B \sim_\mathbb{Q} D$ we have $B \geq N$. Note that $D=P+N$ is a CKM-Zariski decomposition. Since $D$ is abundant, we have that $\mathrm{Fix}\|D\|= N_\sigma(D) \leq N \leq \mathrm{Fix}\|D\|$ by [\ref{Lehmann}, Proposition 6.18] and hence $P=P_\sigma (D)$. Furthermore, we can assume that there exist a smooth projective variety $T$ and a big $\mathbb{Q}$-divisor $G$ on $T$ such that $\mu: X \rightarrow T$ is a contraction and $P_\sigma (D)=P_\sigma(\mu^\ast G)$ by [\ref{Lehmann-II}, Theorem 6.1, Theorem 5.7]. It follows that $Z$ is also the ample model of $G$. Notice that the rational map $h: T \dashrightarrow Z$ is birational. Therefore $H=p_\ast G$ is an $\mathbb{R}$-Cartier $\mathbb{Q}$-divisor and hence $\mathbb{Q}$-Cartier which completes the proof.
\end{proof}

Finally, we obtain the proposition below by combining Proposition \ref{prop:f_g} and Lemma \ref{lem:f_g}.

\begin{prop}\label{prop:f_g_g}
Let $(X,B)$ be a log canonical pair. Assume that $K_X+B$ is $\mathbb{Q}$-Cartier and abundant, and that $(X,B)$ has the log canonical model. Then, $R(K_X+B)$ is finitely generated. Furthermore, for any reflexive sheaf $\mathscr{F}$, $M_{\mathscr{F}}^p(K_X+B)$ is a finitely generated $R(K_X+B)$-module.
\end{prop}

%%%%%%%%%%%%%%%%%%%%%%%%%%%%%%%%%%%%%

\vspace{2cm}

\flushleft{DPMMS}, Centre for Mathematical Sciences,\\
Cambridge University,\\
Wilberforce Road,\\
Cambridge, CB3 0WB,\\
UK\\
email: zh262@dpmms.cam.ac.uk

\vspace{1cm}

\end{document}